\documentclass[a4paper,11pt]{article}

\usepackage{amsmath}
\usepackage{amsfonts}
\usepackage{amssymb}
\usepackage{mathrsfs}
\usepackage{amscd}
\usepackage{pb-diagram}
\usepackage{amsthm}
\usepackage{color}
\usepackage[all]{xy}
\usepackage{graphicx}
\usepackage{url}
\usepackage{enumerate}
\usepackage[titletoc,title]{appendix}
\usepackage{dsfont}

\numberwithin{equation}{section}

\usepackage{tikz}          
\usetikzlibrary{matrix, arrows, decorations.pathmorphing}

\author{Mathieu Molitor
\\
\small{\it{e-mail:}}\,\,\url{pergame.mathieu@gmail.com}}
\title{One-dimensional exponential families with constant Hessian scalar curvature} 
\date{}

\begin{document}

\theoremstyle{definition}
\newtheorem{lemma}{Lemma}[section]
\newtheorem{definition}[lemma]{Definition}
\newtheorem{proposition}[lemma]{Proposition}
\newtheorem{corollary}[lemma]{Corollary}
\newtheorem{theorem}[lemma]{Theorem}
\newtheorem{remark}[lemma]{Remark}
\newtheorem{example}[lemma]{Example}
\bibliographystyle{alpha}

\maketitle 


\begin{abstract}
	We give a complete classification of 1-dimensional exponential families $\mathcal{E}$ defined 
	over a finite space $\Omega=\{x_{0},...,x_{m}\}$ whose Hessian scalar curvature is constant. 
	We observe an interesting phenomenon: if $\mathcal{E}$ has constant 
	Hessian scalar curvature, say $\lambda,$ then $\lambda=\tfrac{2}{k}$ for 
	some positive integer $k\leq m$. We also discuss the central role played by the binomial distribution 
	in this classification. 
\end{abstract}

\section{Introduction}

	Let $\Omega=\{x_{0},...,x_{m}\}$ be a finite set endowed with the counting measure and let $\mathcal{E}$ be 
	a 1-dimensional exponential family defined over $\Omega$, with elements of the form 
	$p(x;\theta)=\textup{exp}(C(x)+\theta F(x)-\psi(\theta))$, where $C,F:\Omega\to \mathbb{R}$ are functions, 
	$\theta\in \mathbb{R}$ and $\psi:\mathbb{R}\to \mathbb{R}$. We denote by $h_{F}$, $\nabla^{(e)}$ and $\nabla^{(m)}$, 
	the Fisher metric, exponential connection and mixture connection, respectively. 

	As it is well-known, the dualistic structure $(h_{F},\nabla^{(e)},\nabla^{(m)})$ is dually flat \cite{Amari-Nagaoka}. 
	Therefore, the tangent bundle $T\mathcal{E}$ is naturally a K\"{a}hler manifold of real dimension 2 
	\cite{Molitor-exponential,shima}. Let $\textup{Scal}:T\mathcal{E}\to \mathbb{R}$ be the corresponding scalar curvature. 

	In this paper, we classify all 1-dimensional exponential families, as described above, for which 
	$\textup{Scal}$ is constant (see Theorem \ref{nekwndkeknskn}). 
	Our proof is based on the particularly simple expression for the Ricci tensor in complex coordinates 
	(since $T\mathcal{E}$ is K\"{a}hler), which implies that $\textup{Scal}$ factorizes as $\textup{Scal}=S\circ \pi$, where 
	$\pi:T\mathcal{E}\to \mathcal{E}$ is the canonical projection and $S:\mathcal{E}\to \mathbb{R}$ is a globally well-defined 
	function. Thus, solving the equation $\textup{Scal}\equiv \textup{constant}$ amounts to solve the simpler 
	equation $S\equiv \textup{constant}$, which can be done by solving elementary differential equations in one variable. 

	An interesting consequence of the above classification is that if $T\mathcal{E}$ has constant scalar curvature, 
	then $\textup{Scal} = \tfrac{2}{k}$ for some positive integer $k$ satisfying 
	$1\leq k\leq m$ (see Corollary \ref{ekwjdkefkjk}). For instance, if 
	$\mathcal{E}=\mathcal{B}(n)$ is the set of binomial distributions defined over $\{0,1,...,n\}$, then 
	$\textup{Scal}=\tfrac{2}{n}$. 

	The last section of the paper is devoted to analysing the ``internal symmetries" of the problem, leading to 
	a somewhat simpler reformulation of the classification discussed above 
	that emphasizes the importance of the binomial distribution. For this purpose, we introduce an equivalence 
	relation $\sim$ on the set $E_{m}$ of all 1-dimensional exponential families defined over the same set 
	$\Omega=\{x_{0},...,x_{m}\}$ by declaring $\mathcal{E}\in E_{m}$ to be equivalent to $\mathcal{E}'\in E_{m}$ if and only if 
	they coincide as spaces of maps $\Omega\to \mathbb{R}$. In Proposition \ref{nekwndknk}, we show that 
	the set of equivalence classes is in one-to-one correspondence with the affine Grassmannian 
	$\textup{Graff}_{1}(\mathbb{R}^{m})$ of 1-dimensional affine subspaces of $\mathbb{R}^{m}$, that is, 
	\begin{eqnarray*}
		 E_{m}/\sim \,\,\cong\,\, \textup{Graff}_{1}(\mathbb{R}^{m}). 
	\end{eqnarray*}
	Then, given $\mathcal{E}\in E_{m}$, we introduce the \textit{reduced exponential family} of $\mathcal{E}$, denoted by 
	$\mathcal{E}_{red}$ (see Definition \ref{nckwdknknk}). 
	It is an exponential family defined over a finite set $\Omega_{red}$, with 
	elements of the form $\textup{exp}\{C_{red}(x)+\theta F_{red}(x)-\psi_{red}(\theta)\}$. Its dualistic structure 
	is isomorphic to that of $\mathcal{E}$, but in general $\Omega_{red}\neq \Omega$, 
	and $F_{red}$ is always strictly increasing. Then we reformulate the classification given 
	in Theorem \ref{nekwndkeknskn} as follows (see Proposition \ref{nkdnkenkknk}). If $\mathcal{E}$ is 
	a 1-dimensional exponential family defined over $\Omega=\{x_{0},...,x_{m}\}$, 
	then $\textup{Scal}:T\mathcal{E}\to \mathbb{R}$ is constant 
	if and only if $\mathcal{E}_{red}\sim \mathcal{B}(p)$, where $p+1$ is the cardinality of $\Omega_{red}$. 

	For the convenience of the reader, the paper gives a rather detailed discussion on the relation between 
	K\"{a}hler geometry and statistics. The topics covered include: definition and 
	examples of K\"{a}hler manifolds (Section \ref{nekkdwnrknfkn}), connections and connectors 
	(Section \ref{pouette}), Dombrowki's construction (Section \ref{nknknfeknaekn}), 
	Ricci curvature (Section \ref{nkndkefknkn}) and statistical manifolds (Section \ref{nfknwkdenknwkn}). 

	In Section \ref{fnekwndkefnknk}, we classify all 1-dimensional exponential families $\mathcal{E}$ with constant 
	scalar curvature on $T\mathcal{E}$ (Theorem \ref{nekwndkeknskn}). 

	In Section \ref{kwdenfknknk}, we reformulate the classification result obtained in the preceding section 
	by using equivalence classes and reduced exponential families (Proposition \ref{nkdnkenkknk}).\\

	\textbf{Notations.} If $M$ is a manifold, then $\mathfrak{X}(M)$ will denote the space of vector fields on 
	$M$ and $C^{\infty}(M)$ the space of smooth real-valued functions on $M$. $TM$ will denote the tangent bundle 
	of $M$ and $TTM$ the tangent bundle of the tangent bundle of $M$ 
	(hence if $\textup{dim}(M)=n$, then $\textup{dim}(TTM)=4n$). The derivative of a smooth map $f:M\to N$ between 
	manifolds at a point $p\in M$ will be denoted by $f_{*_{p}}$. 

\section{K\"{a}hler manifolds}\label{nekkdwnrknfkn}
	General references are \cite{ballmann2006,Huybrechts,moroianu_2007}.
\begin{definition}
	A \textit{complex manifold} of complex dimension $n$ is a Hausdorff topological space $M$ together with 
	a family of maps $\phi_{\alpha}:U_{\alpha}\to\mathbb{C}^{n}$, 
	$\alpha\in\:A$, where each $U_{\alpha}\subset M$ is an open set, such that:
	\begin{itemize}
	\item $\cup_{\alpha\in\:A}$ $U_{\alpha}=M$,
	\item $\phi_{\alpha}(U_{\alpha})$ is an open subset of $\mathbb{C}^{n}$ for all $\alpha\in A$,
	\item $\phi_{\alpha}$ is a homeomorphism onto its image for all $\alpha\in A$,
	\item $\phi_{\beta}\circ\phi_{\alpha}^{-1}: \phi_{\alpha}(U_{\alpha}\cap U_{\beta})\to\phi_{\beta}(U_{\alpha}\cap U_{\beta})$ is 
	holomorphic for all $\alpha ,\beta\in A$ (provided $U_{\alpha}\cap U_{\beta}\neq\emptyset$). 
	\end{itemize}
	The family $\mathcal{A}:=\big\{(U_{\alpha},\phi_{\alpha})\,\vert\,\alpha\in A\big\}$ is called a \textit{complex atlas}. 
\end{definition}
	Just as for smooth manifolds, one defines \textit{complex charts} and \textit{complex coordinates} on a complex manifold $M$.

	In what follows, we will often identify $\mathbb{C}^{n}$ with $\mathbb{R}^{2n}$ via the map 
	\begin{eqnarray*}
	\mathbb{R}^{2n}\rightarrow \mathbb{C}^{n},\,\,\,(x_{1},...,x_{n},y_{1},...,y_{n})\mapsto (x_{1}+iy_{1},...,x_{n}+iy_{n}). 
	\end{eqnarray*}
	Upon this identification, every complex atlas of $M$ determines a smooth atlas of real dimension $2n$. Therefore, 
	complex manifolds of complex dimension $n$ are naturally smooth manifolds of dimension $2n$.

	Let $M$ be a complex manifold of complex dimension $n$ with complex atlas 
	$\mathcal{A}=\big\{(U_{\alpha},\phi_{\alpha})\,\big\vert\,\alpha\in A\big\}$. For each $\alpha\in A$ and each $p\in U_{\alpha}$, 
	define $(J_{\alpha})_{p}\,:\,T_{p}M\rightarrow T_{p}M$ by 
	\begin{eqnarray*}
		(J_{\alpha})_{p}:=(\phi_{\alpha}^{-1})_{*_{\phi(p)}}\circ J_{\mathbb{R}^{2n}}\circ 
		(\phi_{\alpha})_{*_{p}},
	\end{eqnarray*}
	where $J_{\mathbb{R}^{2n}}\,:\,\mathbb{R}^{2n}\to \mathbb{R}^{2n}$ is the linear map 
	whose matrix representation in the canonical basis is 
	\begin{eqnarray*}
		\begin{bmatrix}
			 0      &  -I_{n}\\
			I_{n}   &   0
		\end{bmatrix}. 
	\end{eqnarray*}
	It is easy to check that if $p\in U_{\alpha}\cap U_{\beta}$, then $(J_{\alpha})_{p}=(J_{\beta})_{p}$. Thus we will 
	use the notation $J_{p}$ instead of $(J_{\alpha})_{p}$. 
	Letting $p\in M$ vary, we obtain a smooth tensor $J\,:\,TM\rightarrow TM$ on the smooth manifold $M$ that satisfies $J\circ J=-Id$. 

	The tensor $J$ is called the \textit{complex structure} of the complex manifold $M$. 
\begin{definition}
	Let $M$ be a smooth manifold. A smooth tensor $J\,:\,TM\rightarrow TM$ satisfying $J\circ J=-Id$ is called 
	an \textit{almost complex structure}. 
\end{definition}
\begin{example}
	The complex structure of a complex manifold is an almost complex structure. 
\end{example}

	An almost complex structure $J$ on a smooth manifold $M$ is said to be \textit{integrable} if there 
	exists a complex atlas on $M$ whose corresponding complex structure coincides with $J$. 
	The Newlander-Nirenberg Theorem asserts that an almost complex structure $J$ is integrable if and only if 
	the \textit{Nijenhuis tensor}, 
	defined by
	\begin{eqnarray*}
		N^J(X,Y)=[X,Y]+J[JX,Y]+J[X,JY]-[JX,JY],
	\end{eqnarray*}
	vanishes identically for all vector fields $X,Y$ on $M$ (see \cite{Newlander}). 

	Therefore, a complex manifold can be viewed as a pair $(M,J)$, where $M$ is a smooth manifold and $J\,:\,TM\rightarrow TM$ is 
	an integrable almost complex structure. 
\begin{definition}
	An \textit{almost Hermitian manifold} is a triple $(M,g,J)$, where $M$ is a smooth manifold, $g$ is a 
	Riemannian metric and $J$ is an almost complex structure such that $g_{p}(Ju,Jv)=g_{p}(u,v)$ 
	for all $p\in M$ and all $u,v\in T_{p}M$.
\end{definition}
	If $(M,g,J)$ is an almost Hermitian manifold, we define a 2-form $\omega$ on $M$, called the \textit{fundamental form}, by 
	\begin{eqnarray*}
		\omega_{p} (u,v):= g_{p}(Ju,v),\quad (p\in M,\,\,\,u,v\in T_{p}M).
	\end{eqnarray*}
\begin{definition} 
	A \textit{K\"{a}hler manifold} is an almost Hermitian manifold $(M,g,J)$ satisfying the following 
	analytical conditions:
	\begin{enumerate}[(i)]
	\item the fundamental form $\omega$ is closed, that is, $d\omega =0$,
	\item $J$ is integrable.
	\end{enumerate}
\end{definition}
\begin{example}\label{exa:3.2} 
	The manifold $M=\mathbb{\mathbb{C}}^{n}\simeq\mathbb{R}^{2n}$ endowed with the Euclidean metric and the 
	almost complex structure 
	$J_{\mathbb{R}^{2n}}$ is a K\"{a}hler manifold, whose fundamental form is
	\begin{align*}
		\omega=\sum_{k=1}^{n}dx_{k}\land dy_{k},
	\end{align*}
	where $x_{1},...x_{n},y_{1},...,y_{n}$ are linear coordinates of $\mathbb{R}^{2n}$. 
 \end{example}
\begin{example}\label{nkdnfkekwnk}
	The unit sphere $S^{2}\subseteq \mathbb{R}^{3}$ endowed with the round metric $g$ induced by $\mathbb{R}^{3}$, 
	and complex structure $J_{p}(u):=-p\times u$ (cross product) is a K\"{a}hler manifold with fundamental 
	form $\omega$ given by 
	\begin{eqnarray*}
		 \omega_{p}(u,v):=-\textup{det}(p,u,v)=-g_{p}(p,u\times v),
	\end{eqnarray*}
	where $p\in S^{2}$ and $u,v\in T_{p}S^{2}=\{\textup{plane orthogonal to}\,p\}$. 
\end{example}
\begin{example}\label{exa:3.4} 
	The complex projective space $\mathbb{P}(\mathbb{C}^{n})$ is the set of all complex lines in 
	$\mathbb{C}^{n}$ passing through the origin. Let $\pi\,:\,\mathbb{C}^{n}-\{0\}\rightarrow \mathbb{P}(\mathbb{C}^{n}),\,\,\,
	z=(z_{1},...,z_{n})\mapsto 
	[z]=[z_{1},...,z_{n}]=\mathbb{C}z$. We define a topology on $\mathbb{P}(\mathbb{C}^{n})$ 
	by declaring $U\subseteq \mathbb{P}(\mathbb{C}^{n})$ 
	to be open if and only if $\pi^{-1}(U)$ is open in $\mathbb{C}^{n}-\{0\}$. 
	It can be shown that the family of maps $\phi_{i}\,:\,\big\{[z_{1},...,z_{n}]\,
	\big\vert\,z_{i}\neq 0\big\}\rightarrow \mathbb{C}^{n-1},\,\,\,
	[z_{1},...,z_{n}]\mapsto \big(\tfrac{z_{1}}{z_{i}},...,
	\tfrac{z_{i-1}}{z_{i}},\tfrac{z_{i+1}}{z_{i}},...,\tfrac{z_{n}}{z_{i}}\big)$, 
	$i=1,...,n$, defines a complex atlas on $\mathbb{P}(\mathbb{C}^{n})$. The restriction of $\pi$ to the unit sphere 
	$S^{2n-1}\subseteq \mathbb{R}^{2n}\cong \mathbb{C}^{n}$ yields a surjective submersion 
	$\pi\vert_{S^{2n-1}}\,:\,S^{2n-1}\rightarrow \mathbb{P}(\mathbb{C}^{n})$ and hence there are tensors $g$ and $\omega$ 
	on $\mathbb{P}(\mathbb{C}^{n})$ characterized by the formulas 
	\begin{eqnarray*}
		(\pi\vert_{S^{2n-1}})^{*}g=j^{*}\textup{Re} \langle \cdot,\cdot \rangle \quad \textup{and}
		\quad(\pi\vert_{S^{2n-1}})^{*}\omega=j^{*}\textup{Im} \langle \cdot,\cdot \rangle,
	\end{eqnarray*}
	where $j: S^{2n-1}\hookrightarrow\mathbb{C}^{n}\cong \mathbb{R}^{2n}$ is the inclusion, 
	$\textup{Re}\langle \cdot,\cdot \rangle$ and $\textup{Im}\langle \cdot,\cdot\rangle$ are the real and imaginary parts of 
	the standard Hermitian product $\langle z,w\rangle=\overline{z}_{1}w_{1}+...+\overline{z}_{n}w_{n}$ on $\mathbb{C}^{n}$.

	It can be shown that $(\mathbb{P}(\mathbb{C}^{n}),g,J)$ is a K\"{a}hler manifold, where $J$ 
	is the associated complex structure, with fundamental form $\omega$. 
\end{example}

\section{Connections and connectors}\label{pouette}

	This section follows closely \cite{Dombrowski}. Let $M$ be a manifold. 
\begin{definition}\label{def:2.1}
	A \textit{linear connection} $\nabla$ on $M$ is a map 
	$\mathfrak{X}(M)\times\mathfrak{X}(M)\to\mathfrak{X}(M)$, $(X,Y)\mapsto\nabla_{X}Y$, 
	satisfying the following properties:
	\begin{enumerate}[(i)]
		\item $\nabla_{fX+gY}Z=f\nabla_{X}Z+g\nabla_{Y}Z$,
		\item $\nabla_{X}(Y+Z)=\nabla_{X}Y+\nabla_{X}Z$,
		\item $\nabla_{X}(fY)=X(f)Y+f\nabla_{X}Y$,
	\end{enumerate}
	for all vector fields $X,Y,Z\in\mathfrak{X}(M)$ and for all functions 
	$f,g\in C^{\infty}(M)$. 
\end{definition}
	
%
%
%
	
	In local coordinates $(x_{1},...,x_{n})$ on $U\subseteq M$, if $X=\sum_{k=1}^{n}X^{i}\tfrac{\partial}{\partial x_{k}}$ 
	and $Y=\sum_{k=1}^{n}Y^{k}\tfrac{\partial}{\partial x_{k}}$, then, by standard computations, 
	\begin{eqnarray}\label{nekwmdkvnksnk}
		\nabla_{X}Y=\sum_{k=1}^{n}\bigg(X(Y^{k})+\sum_{i,j=1}^{n}X^{i}Y^{j}\Gamma_{ij}^{k}\bigg)\dfrac{\partial}{\partial x_{k}},
	\end{eqnarray}
	where $\Gamma_{ij}^{k}:U\to \mathbb{R}$ are the Christoffel symbols, defined by the formula 
	\begin{eqnarray*}
		\nabla_{\frac{\partial}{\partial x_{i}}}\frac{\partial}{\partial x_{j}}
		=\sum_{k=1}^n\Gamma_{ij}^{k}\frac{\partial}{\partial x_{k}}\,\,\,\,\,\,\qquad \textup{for}\,\,i,j=1,...,n.
	\end{eqnarray*}

	Let $\pi:TM\to M$ be the canonical projection and let 
	$(U,\varphi)$ be a chart for $M$ with local coordinates $(x_{1},...,x_{n})$. Define 
	$\widetilde{\varphi}:\pi^{-1}(U)\to \mathbb{R}^{2n}$ by
	\begin{eqnarray*}
		\widetilde{\varphi}\bigg(\sum_{i=1}^{n}u_{i}\dfrac{\partial}{\partial x_{i}}\bigg\vert_{p}\bigg)
		=(x_{1}(p),...,x_{n}(p),u_{1},...,u_{n}). 
	\end{eqnarray*}
	Then $(\pi^{-1}(U),\widetilde{\varphi})$ is a chart for $TM$; let 
	$(q_{1},...,q_{n},r_{1},...,r_{n})$ be the corresponding local coordinates (in particular, 
	$q_{i}=x_{i}\circ \pi$ for every $i=1,...,n)$.

	Let $u=\sum_{k=1}^{n}u_{k}\tfrac{\partial}{\partial x_{k}}\big\vert_{p}\in \pi^{-1}(U)$ be arbitrary. 
	Define a linear map $K_{u}:T_{u}(TM)\to T_{p}M$ by 
	\begin{eqnarray*}
		&&K_{u}\bigg(\dfrac{\partial}{\partial q_{a}}\bigg\vert_{u}\bigg):=
			\sum_{k,j=1}^{n}\Gamma_{aj}^{k}(p)u_{j}\dfrac{\partial}{\partial x_{k}}\bigg\vert_{p}
			 \,\,\,\,\,\,\,\,\,\,\textup{for}\,\,a=1,...,n,\\[0.9em]
		&&K_{u}\bigg(\dfrac{\partial}{\partial r_{a}}\bigg\vert_{u}\bigg):=
			\dfrac{\partial}{\partial x_{a}}\bigg\vert_{p}\,\,\,\,\,\,\,\,\,\,\textup{for}\,\,a=1,...,n.
	\end{eqnarray*}
\begin{lemma}\label{nfkewefnknknk}
	Let $X$ and $Y$ be vector fields on $M$. Suppose $Y(p)=u$. Then $K_{u}(Y_{*_{p}}X_{p})=(\nabla_{X}Y)(p)$.
\end{lemma}
\begin{proof}
	By standard computations, 
	\begin{eqnarray*}
		Y_{*_{p}}X_{p}=\sum_{a=1}^{n}\bigg(X^{a}(p)\dfrac{\partial}{\partial q_{a}}\bigg\vert_{u}
		+X_{p}(Y^{a})\dfrac{\partial}{\partial r_{a}}\bigg\vert_{u}\bigg)
	\end{eqnarray*}
	so 
	\begin{eqnarray*}
		K_{u}(Y_{*_{p}}X_{p})=
		\sum_{a=1}^{n}\bigg(X^{a}(p)\sum_{k,j=1}^{n}\Gamma_{aj}^{k}(p)Y^{j}(p)\dfrac{\partial}{\partial x_{k}}\bigg\vert_{p}
			+X_{p}(Y^{a}) \dfrac{\partial}{\partial x_{a}}\bigg\vert_{p}\bigg),
	\end{eqnarray*}
	where we have used $u_{j}=Y^{j}(p)$. Comparing the above formula with the local 
	expression for $\nabla_{X}Y$ in coordinates (see \eqref{nekwmdkvnksnk}), one obtains the desired formula. 
\end{proof}

	Clearly, vectors of the form $Y_{*_{p}}X_{p}$, with $Y_{p}=u$, generate $T_{u}(TM)$, and so the above lemma implies 
	that the definition of $K_{u}$ is independant of the choice of the chart $(U,\varphi)$. 
	
	The map 
	\begin{eqnarray*}
		K:TTM\to TM,
	\end{eqnarray*}
	defined for $A\in T_{u}(TM)$ by $K(A):=K_{u}(A)$, is called \textit{connector}, or \textit{connection map}, 
	associated to $\nabla$. 

	The following result is an immediate consequence of the definition of $K$.
\begin{proposition}\label{prop:2.3}
	Let $K$ be the connector associated to a connection $\nabla$ on $M$. The following holds.
	\begin{enumerate}[(i)]
	\item For every pair $X,Y$ of vector fields on $M$, $\nabla_{X}Y=KY_{*}X$,
		where $Y_{*}X$ denotes the derivative of $Y$ in the direction of $X$.
	
	\item For every $u\in T_{p}M$, the restriction of $K$ to $T_{u}(TM)$ is a linear map $T_{u}(TM)\to T_{p}M$.
	\end{enumerate}
\end{proposition}

	If $A\in T_{u}(TM)$ is such that $\pi_{*_{u}}A=0$ and $K(A)=0$, then a simple calculation using local coordinates 
	shows that $A=0$. Therefore, 
\begin{proposition}\label{cor:2.7}
	Let $K$ be the connector associated to a connection $\nabla$ on $M$. 
	Given $u\in T_{p}M$, the map $T_{u}(TM)\to T_{p}M\oplus T_{p}M$,
	defined by
	\begin{eqnarray}
		A\mapsto(\pi_{*_{u}}A,KA),\label{eq:2.5}
	\end{eqnarray}
	is a linear bijection.
\end{proposition}
	Thus, given a linear connection $\nabla$, we can identify at any point $u\in T_{p}M$ the vector spaces  $T_{u}(TM)$
	and $T_{p}M\oplus T_{p}M$ via the map (\ref{eq:2.5}). 

\section{Dombrowski's construction}\label{nknknfeknaekn}

	Let $M$ be a smooth manifold endowed with a connection $\nabla$. We will denote by 
	$\pi:TM\rightarrow M$ the canonical projection. 

	By Proposition \ref{cor:2.7}, there is an identification of vector spaces 
	$T_{u}(TM)\cong T_{p}M\oplus T_{p}M$, where $p=\pi(u)$. If there is no danger of confusion, 
	we will therefore regard an element of $T_{u}(TM)$ as a pair $(v,w)$, where $v,w\in T_{p}M$.

	Let $h$ be a Riemannian metric on $M$. The pair $(h,\nabla)$ determines an almost Hermitian 
	structure on $TM$ via the following formulas:

	\begin{alignat*}{5}\label{equation definition G, omega, etc.}
		g_{u}\big(\big(v,w\big),
			\big(\overline{v},
			\overline{w}\big)\big)\quad:=&\quad
			h_{p}\big(v,\overline{v}\big)+
			h_{p}\big(w,\overline{w}\big),&\textup{}\,\,\,\,\,\,\,\,\,
			(\textup{metric})&\nonumber\\
		\omega_{u}\big(\big(v,w\big),
			\big(\overline{v},
			\overline{w}\big)\big)\quad:=&\quad h_{p}\big(v,\overline{w}\big)-
			h_{p}\big(w,\overline{v}\big),&\textup{}\,\,\,\,\,\,\,\,\,
			(\textup{2-form})&\nonumber\\
		J_{u}\big(\big(v,w\big)\big)\quad:=&\quad
		h_p\big(-w,v\big),&
		(\textup{almost complex structure})
	\end{alignat*}
	where $u,v,w,\overline{v},\overline{w}\;\in\; T_{p}M$.

	The tensors $g,J,\omega$ are smooth (this will follow from their coordinate representation, 
	see Proposition \ref{prop:4.1} below) and clearly, 
	$J^{2}=-Id$, $g(Ju,Jv)=g(u,v)$ and $\omega(u,v)=g(Ju,v)$ for all $u,v\in TM$ such that $\pi(u)=\pi(v)$. 
	Thus, $(TM,g,J)$ is an almost Hermitian manifold with fundamental form $\omega$. 
	This is \textit{Dombrowski's construction} \cite{Dombrowski}. 

	We now review the analytical properties of Dombrowski's construction. We begin with some definitions.
\begin{definition}\label{def:3.6} 
	A \textit{dualistic structure} on a manifold $M$ is a triple $(h,\nabla,\nabla^*)$, where $h$ is a 
	Riemannian metric and where $\nabla$ and $\nabla^{*}$ are linear connections satisfying 
	\begin{eqnarray*}
		Xg(Y,Z)=g(\nabla_{X}Y,Z)+g(Y,\nabla_{X}^{*}Z)
	\end{eqnarray*}
	for all vector fields $X,Y,Z$ on $M$. The connection $\nabla^{*}$ is called 
	the \textit{dual connection} of $\nabla$ (and vice versa).
\end{definition}
	As the literature is not uniform, let us agree that the torsion $T$ and the curvature tensor $R$ of a
	connection $\nabla$ are defined as
	\begin{alignat*}{1}
		 T(X,Y)&:=\nabla_{X}Y-\nabla_{Y}X-[X,Y],\nonumber\\
		 R(X,Y)Z &:=\nabla_{X}\nabla_{Y}Z-\nabla_{Y}\nabla_{X}Z-\nabla_{[X,Y]}Z,\nonumber
	\end{alignat*}
	where $X,Y,Z$ are vector fields on $M$. By definition, a linear connection is \textit{flat} 
	if the torsion and curvature tensor are identically zero on $M$. A manifold endowed with a 
	flat linear connection is called an \textit{affine manifold}. 

\begin{definition}\label{def:3.7} 
	A dualistic structure $(h,\nabla,\nabla^{*})$ is \textit{dually flat} if both $\nabla$ and $\nabla^{*}$ are flat.
\end{definition}
\begin{proposition}\label{prop:3.8} 
	Let $(h,\nabla,\nabla^{*})$ be a dualistic structure on $M$ and let $(g,J,\omega)$ be 
	the almost Hermitian structure on $TM$ associated to $(h,\nabla)$ via Dombrowski's construction. 
	The following are equivalent.
	\begin{enumerate}[(i)]
	\item $(TM,g,J,\omega)$ is a K\"{a}hler manifold.
	\item $(M,h,\nabla,\nabla^{*})$ is dually flat.
	\end{enumerate}
\end{proposition}
\begin{proof}
	See \cite{Dombrowski,Molitor-exponential}.
\end{proof}

	We now direct our attention to the coordinate expressions for $g,J$ and $\omega$. 

\begin{definition}
	Suppose $(M,\nabla)$ is an affine manifold. An \textit{affine coordinate system} is 
	a coordinate system $(x_{1},...,x_{n})$ defined on some open set $U\subseteq M$ such that 
	\begin{eqnarray*}
		 \nabla_{\tfrac{\partial}{\partial x_{i}}}\dfrac{\partial}{\partial x_{j}}=0
	\end{eqnarray*}
	for all $i,j=1,...,n$. 
\end{definition}

	It can be shown that for every point $p$ in an affine manifold $M$, there is an affine coordinate system 
	$(x_{1},...,x_{n})$ defined on some neighborhood $U\subseteq M$ of $p$ (see \cite{shima}). 
\begin{proposition}\label{prop:4.1}
	Let $(h,\nabla,\nabla^{*})$ be a dually flat structure on a manifold $M$ and let $(g, J,\omega)$ 
	be the K\"{a}hler structure on $TM$ associated to $(h,\nabla)$ via Dombrowski's construction. 
	Let $x=(x_{1},...,x_{n})$ be an affine coordinate system with respect to $\nabla$ on $U\subseteq M$, 
	and let $(q,r)=(q_{1},...,q_{n},r_{1},...,r_{n})$ denote the corresponding coordinates on $\pi^{-1}(U)$, 
	as described before Lemma \ref{nfkewefnknknk}. Then, in the coordinates $(q,r)$,
	\begin{equation}\label{eq:4.2}
		g=\begin{bmatrix}h_{ij} & 0\\
		0 & h_{ij}
		\end{bmatrix},\quad J=
		\begin{bmatrix}0 & -I_{n}\\
		I_{n} & 0
		\end{bmatrix},
		\quad\omega=
		\begin{bmatrix}0 & h_{ij}\\
		-h_{ij} & 0
		\end{bmatrix},
	\end{equation} 
	where $h_{ij}=h\big (\tfrac{\partial}{\partial x_{i}},\tfrac{\partial}{\partial x_{j}}\big)$,
	$i,j=1,...,n$. 
\end{proposition}
\begin{proof}
	See \cite{Molitor2014}.
\end{proof}
\begin{corollary}
	Under the hypotheses of Proposition \ref{prop:4.1}, if $z_{k}:=q_{k}+ir_{k}$, $k=1,...,n$, 
	then $(z_{1},...,z_{n})$ are complex coordinates on the complex manifold $(TM,J)$. 
\end{corollary}

\section{Ricci curvature}\label{nkndkefknkn}

	Let $N$ be a K\"{a}hler manifold with K\"{a}hler metric $g$. We denote by Ric the Ricci tensor of $g$,
	\begin{eqnarray*}
		\textup{Ric}(X,Y):=\textup{trace}\{ Z\mapsto R(Z,X)Y\} ,
	\end{eqnarray*}
	where $X,Y, Z$ are vector fields on $N$ and $R$ is the curvature tensor of $g$.

	On the complexified tangent bundle $TN^{\mathbb{C}}=TN\otimes_{\mathbb{R}}\mathbb{C}$, 
	we extend $\mathbb{C}$-linearly every tensor of $N$ at every point $p\;\in\;N$. 
	For simplicity, we use the same symbols ($g$, Ric, etc) to indicate the corresponding $\mathbb{C}$-linear extensions.

	Regarding local computations and indices, Greek indices $\alpha, \beta, \gamma$ shall run over $1,..., n$ 
	while capital letters $A, B, C, ...$ shall run over $1,... ,n,\overline{1},\overline{2},...,\overline{n}$. 
	Let $(z_1,..., z_n)$ be a system of complex coordinates on $N$. We denote by $x_{\alpha}$ and $y_{\alpha}$ the 
	real and imaginary part of $z_{\alpha}$, i.e., $z_{\alpha}= x_{\alpha}+iy_{\alpha}$. With this notation, the vectors
	\begin{displaymath}
		Z_{\alpha}:=\frac{\partial}{\partial z_{\alpha}}=\frac{1}{2}\biggl(\frac{\partial}{\partial x_{\alpha}}
		-i\frac{\partial}{\partial y_{\alpha}}\biggr)\qquad 
		\textup{and}\qquad \overline{Z}_{\alpha}:=\frac{\partial}{\partial \overline{z}_{\alpha}}
		=\frac{1}{2}\biggl(\frac{\partial}{\partial x_{\alpha}}+i\frac{\partial}{\partial y_{\alpha}}\biggr),
	\end{displaymath}
	where $\alpha=1,...,n$, form a basis for $TN^{\mathbb{C}}$. Let $\textup{Ric}_{AB}=\textup{Ric}(Z_{A},Z_{B})$ 
	be the components of the Ricci tensor in this basis. As it is well-known, these components 
	are elegantly expressed via the following formulas:
	\begin{equation}
	\textup{Ric}_{\alpha\beta}=\textup{Ric}_{\overline{\alpha}\overline{\beta}}\equiv0,
	\quad\textup{Ric}_{\alpha\overline{\beta}}
	=\overline{\textup{Ric}_{\overline{\alpha}\beta}}\qquad \textup{and}\qquad\textup{Ric}_{\alpha\overline{\beta}}
	=-\frac{{\displaystyle \partial^{2}\ln\det(G)}}{\partial z_{\alpha}\partial\overline{z}_{\beta}},\label{eq:4.1}
	\end{equation}
	where $G:=(g_{\alpha\overline{\beta}})_{1\leq\alpha,\beta\leq n}$ is the associated Hermitian matrix. \\

	We now specialize to the case $N = TM$, assuming that $g$ is the K\"{a}hler metric associated to a
	dually flat structure $(h,\nabla,\nabla^{*})$ on $M$ via Dombrowski's construction.

	Fix an affine coordinate system $(x_{1},...,x_{n})$ on an open set $U\subseteq M$ with respect to $\nabla$, and 
	let $(q,r)=(q_{1},...,q_{n},r_{1},...,r_{n})$ be the corresponding coordinates on $\pi^{-1}(U)$, as described 
	before Lemma \ref{nfkewefnknknk}, where $\pi:TM\to M$ is the canonical projection.

	Given $1\leq\alpha\leq n$, define $z_{\alpha}:=q_{\alpha}+ir_{\alpha}$. Then $(z_{1},...,z_{n})$ are complex coordinates 
	on $\pi^{-1}(U)\subseteq TM$. Applying \eqref{eq:4.1}, we obtain
	\begin{equation}
		\label{eq:4.3} g_{\alpha\overline{\beta}}=
		\frac{1}{2}h_{\alpha\beta}\circ\pi\qquad\textup{and}\qquad\textup{Ric}_{\alpha\overline{\beta}}
		=-\frac{1}{4}\biggl(\frac{{\displaystyle \partial^{2}\ln d}}{\partial x_{\alpha}\partial x_{\beta}}\biggr)
		\circ\pi,
	\end{equation}
	where $d$ is the determinant of the matrix $(h_{\alpha\beta})$. 
	The second formula in (\ref{eq:4.3}) is the local expression for the Ricci tensor 
	in the basis $\{Z_{\alpha}, \overline{Z}_{\alpha}\}$. Returning to the coordinates 
	$(q,r)$, a direct calculation using

	\begin{eqnarray*}
		\dfrac{\partial}{\partial q_{k}}=\dfrac{\partial}{\partial z_{k}}+\frac{\partial}{\partial\overline{z}_{k}}
		\qquad\textup{and}\qquad\dfrac{\partial}{\partial r_{k}}=
		i\biggl(\frac{\partial}{\partial z_{k}}-\frac{\partial}{\partial\overline{z}_{k}}\biggr),
	\end{eqnarray*}
	shows the following result (see \cite{Molitor2014}). 

\begin{proposition}\label{prop:4.2}
	Let $(h,\nabla,\nabla^{*})$ be a dually flat structure on $M$ and let $g$ be the K\"{a}hler metric on $TM$ 
	associated to $(h,\nabla)$ via Dombrowski's construction. If $x = (x_1,..., x_n)$ is an affine coordinate system 
	on $M$ with respect to $\nabla$, then in the coordinates $(q,r)$, the matrix representation of the Ricci tensor of $g$ is

	\begin{eqnarray}
		\textup{Ric}\bigl(q,r\bigr)=
		\begin{bmatrix}
			\beta_{\alpha\beta} & 0\\
			0 & \beta_{\alpha\beta}
		\end{bmatrix},
		\qquad\textup{where}\quad\:
		\beta_{\alpha\beta}=
		-\frac{1}{2}\frac{{\displaystyle \partial^{2}\ln d}}{\partial x_{\alpha}\partial x_{\beta}},
	\end{eqnarray}
	and where $d$ is the determinant of the matrix $h_{\alpha\beta}=h\big(\tfrac{\partial}{\partial x_{\alpha}}
	,\tfrac{\partial}{\partial x_{\alpha}}\big)$.
\end{proposition}
	Recall that the scalar curvature is, by definition, the trace of the Ricci tensor.
\begin{corollary}\label{cor:4.3}
	Under the hypotheses of Proposition \ref{prop:4.2}, the scalar curvature of $g$ is given in the coordinates $(q,r)$ 
	by 
	\begin{eqnarray}\label{eq:4.4}
	\textup{Scal}(q,r)=-\sum_{\alpha,\beta=1}^{n}h^{\alpha\beta}
		\frac{\partial^2\ln d}{\partial x_{\alpha}\partial x_{\beta}},
	\end{eqnarray}
	where $d$ is the determinant of the matrix $h_{\alpha\beta}$, and where 
	$h^{\alpha\beta}$ are the coefficients of the inverse matrix of $h_{\alpha\beta}$.
\end{corollary}

	Observe that the scalar curvature on $TM$ can be written $\textup{Scal}=S\circ\pi$, where
	$S:M\to \mathbb{R}$ is a globally defined function whose local expression is given by the right hand side 
	of \eqref{eq:4.4}. The function $S$ is called \textit{Hessian scalar curvature} (see \cite{shima}). 

\section{Statistical manifolds}\label{nfknwkdenknwkn}

	General references are \cite{Jost2,Amari-Nagaoka,Murray}. 
\begin{definition}\label{def:5.1}
	A \textit{statistical manifold} is a pair $(S,j)$, where $S$ is a manifold and where $j$ 
	is an injective map from $S$ to the space of all probability density functions $p$ 
	defined on a fixed measure space $(\Omega,dx)$: 
	\begin{eqnarray*}
		j:S\hookrightarrow\Bigl\{ p:\Omega\to\mathbb{R}\;\Bigl|\; p\:\textup{is measurable, }p\geq 
		0\textup{ and }\int_{\Omega}p(x)dx=1\Bigr\}.
	\end{eqnarray*}
\end{definition}

	If $\xi=(\xi_{1},...,\xi_{n})$ is a coordinate system on a statistical manifold $S$, then we shall 
	indistinctly write $p(x;\xi)$ or $p_{\xi}(x)$ for the probability density function determined by $\xi$.

	Given a ``reasonable" statistical manifold $S$, it is possible to define a metric $h_{F}$ and a 
	family of connections $\nabla^{(\alpha)}$ on $S$ $(\alpha\in\mathbb{R})$ in the following way: 
	for a chart $\xi=(\xi_{1},...,\xi_{n})$ of $S$, define
	\begin{alignat*}{1}
		\bigl(h_F\bigr)_{\xi}\bigl(\partial_{i},\partial_{j}\bigr) & :=
		\mathbb{E}_{p_{\xi}}\bigl(\partial_{i}\ln\bigl(p_{\xi}\bigr)\cdotp\partial_{j}\ln\bigl(p_{\xi}\bigr)\bigr),
		\nonumber\\
		\Gamma_{ij,k}^{(\alpha)}\bigl(\xi\bigr) & :=
		\mathbb{E}_{p_{\xi}}\bigl[\bigl(\partial_{i}\partial_{j}\ln\bigl(p_{\xi}\bigr)
		+\tfrac{1-\alpha}{2}\partial_{i}\ln\bigl(p_{\xi}\bigr)\cdotp\partial_{j}\ln\bigl(p_{\xi}\bigr)\bigr),
		\partial_{k}\ln\bigl(p_{\xi}\bigr)\bigr],\label{eq:48}\nonumber
	\end{alignat*}
	where $\mathbb{E}_{p_{\xi}}$ denotes the mean, or expectation, with respect to the probability 
	$p_{\xi}dx$, and where $\partial_{i}$ is a shorthand for $\tfrac{\partial}{\partial\xi_{i}}$. 
	It can be shown that if the above expressions are defined and smooth for every chart of $S$, 
	then $h_F$ is a well defined metric on $S$ called the \textit{Fisher metric}, and that the 
	$\Gamma_{ij,k}^{(\alpha)}$'s define a connection $\nabla^{(\alpha)}$ via the formula 
	$\Gamma_{ij,k}^{(\alpha)}(\xi)=(h_F)_{\xi}(\nabla_{\partial_{i}}^{(\alpha)}\partial_{i},\partial_{j})$, 
	which is called the \textit{$\alpha$-connection}.

	Among the $\alpha$-connections, the $(\pm1)$-connections are particularly important; the 1-connection is
	usually referred to as the \textit{exponential connection}, also denoted by $\nabla^{(e)}$, while the 
	$(−1)$-connection is referred to as the \textit{mixture connection}, denoted by $\nabla^{(m)}$.

	In this paper, we will only consider statistical manifolds $S$ for which the Fisher metric and
	$\alpha$-connections are well defined.

\begin{proposition}\label{prop:5.2}
	Let $S$ be a statistical manifold. Then, $(h_F,\nabla^{(\alpha)},\nabla^{(-\alpha)})$ is a dualistic structure on $S$. 
	In particular, $\nabla^{(-\alpha)}$ is the dual connection of $\nabla^{(\alpha)}$.
\end{proposition}
\begin{proof}
	See \cite{Amari-Nagaoka}.
\end{proof}

	We now recall the definition of an exponential family. 
\begin{definition}\label{def:5.3} 
	An \textit{exponential family} $\mathcal{E}$ on a measure space $(\Omega,dx)$ is a set of probability 
	density functions $p(x;\theta)$ of the form
	\begin{eqnarray*}
		p(x;\theta)=\exp\biggl\{ C(x)+\sum_{i=1}^{n}\theta_{i}F_{i}(x)-\psi(\theta)\biggr\},\label{eq:49}
	\end{eqnarray*}
	where $C,F_1...,F_n$ are measurable functions on $\varOmega$, $\theta=(\theta_{1},...,\theta_{n})$ 
	is a vector varying in an open subset $\Theta$ of $\mathbb{R}^{n}$ and where $\psi$ is a function defined on $\Theta$.
\end{definition}
	In the above definition it is assumed that the family of functions $\{1,F_1,...,$ $F_n\}$ 
	is linearly independent, so that the map $p(x,\theta)\mapsto\theta$ becomes a bijection, hence defining a global 
	chart for $\mathcal{E}$. The parameters $\theta_{1},...,\theta_{n}$ are called the 
	\textit{natural} or \textit{canonical parameters} of the exponential family $\mathcal{E}$.
 
\begin{example}[\textbf{Normal distribution}]\label{exa:5.4} 
	Normal distributions,
	\begin{eqnarray*}
		p(x;\mu,\sigma)=\frac{1}{\sqrt{2\pi}\sigma}\exp\biggl\{-\frac{\bigl(x-\mu\bigr)^{2}}{2\sigma^{2}}\biggr\}
		\quad\bigl(x\in\mathbb{R}\bigr),
	\end{eqnarray*}
	form a $2$-dimensional statistical manifold, denoted by $\mathcal{N}$, parameterized by 
	$(\mu,\sigma)\in\mathbb{R}\times\mathbb{R}_{+}^{*}$, where $\mu\in\mathbb{R}$ is the mean and 
	$\sigma\in\mathbb{R}_{+}^{*}$ is the standard deviation 
	(here $\mathbb{R}_{+}^{*}:=\bigl\{ x\in\mathbb{R}\bigl|x>0\bigr\}$). It is an exponential family, because 
	$p(x;\mu,\sigma)=\exp\big\{\theta_{1}F_{1}(x)+\theta_{2}F_{2}(x)-\psi(\theta)\big\}$, where
	\begin{align*}
	 & \theta_{1}=\frac{\mu}{\sigma^{2}},\quad\theta_{2}=-\frac{1}{2\sigma^{2}},\quad C(x)=0,
	\quad F_{1}(x)= x,\quad F_{2}(x)= x^{2},\\
	 & \psi(\theta)=-\frac{(\theta_{1}\bigr)^{2}}{4\theta_{2}}+\frac{1}{2}\ln\Big(-\frac{\pi}{\theta_{2}}\Big).
	\end{align*}
\end{example}
\begin{example}\label{exa:5.5}
	Given a finite set $\Omega=\{ x_{1},...,x_{n}\}$, define
 	\begin{eqnarray*}
		\mathcal{P}_{n}^{\times}\:=\biggl\{ p:\Omega\to\mathbb{R}\:\Bigl|\: p(x)>0 
		\textup{ for all }x\in\Omega\textup{ and }\sum_{k=1}^{n}p(x_{k})=1\biggr\}.
	\end{eqnarray*}
	Elements of $\mathcal{P}_{n}^{\times}$ can be parametrized as follows: 
	$p(x;\theta)=\exp\big\{\sum_{i=1}^{n-1}\theta_{i}F_{i}(x)-\psi(\theta)\big\}$,
	where 
	\begin{eqnarray*}
		&&\theta=(\theta_{1},...,\theta_{n-1})\in\mathbb{R}^{n-1},\,\,\,\,\,\,
		F_{i}(x_{j})=\delta_{ij}\,\,\,\,\,\,\textup{(Kronecker delta)},\\
		&& \psi(\theta)=-\ln\Big(1+\sum_{i=1}^{n-1}\exp\bigl(\theta_{i}\bigr)\Big). 
	\end{eqnarray*}
	Therefere $\mathcal{P}_{n}^{\times}$ is an exponential family of dimension $n-1$. 
\end{example}
\begin{example}[\textbf{Binomial distribution}]\label{exa:5.6}
	The set of binomial distributions defined over $\Omega:=\{0,...,n\}$,
	\begin{eqnarray*}
		p(k)=\binom{n}{k}q^{k}\bigl(1-q\bigr)^{n-k},\quad (k\;\in\;\Omega,\;q\in(0,1)),
	\end{eqnarray*}
	where $\binom{n}{k}=\tfrac{n!}{(n-k)!k!}$, is a $1$-dimensional statistical manifold, 
	denoted by $\mathcal{B}(n)$, parametrized 
	by $q\in\bigl(0,1\bigr)$. It is an exponential family, because 
	$p(k)=\exp\big\{ C(k)+\theta F(k)-\psi(\theta)\big\}$,
	where
	\begin{eqnarray*}
		&&\theta=\ln\Bigl(\frac{q}{1-q}\Bigr),\quad C\bigl(k\bigr)=
		\ln\binom{n}{k},\quad F(k)= k,\quad\\
		&& \psi(\theta)= n\ln\bigl(1+\exp(\theta)\bigr).
	\end{eqnarray*}
\end{example}

\begin{proposition}\label{prop:5.7}
	Let $\mathcal{E}$ be an exponential family such as in Definition \ref{def:5.3}. Then 
	$(\mathcal{E},h_F,\nabla^{(e)},\nabla^{(m)})$ is dually flat. 
\end{proposition}
\begin{proof}
	See \cite{Amari-Nagaoka}. 
\end{proof}
\begin{corollary}\label{cor:5.8} 
	The tangent bundle $T\mathcal{E}$ of an exponential family $\mathcal{E}$ is a K\"{a}hler 
	manifold for the K\"{a}hler structure $(g,J,\omega)$ associated to $(h_F,\nabla^{(e)})$ via Dombrowski's 
	construction.
\end{corollary}
\begin{proof}
	Follows from Proposition \ref{prop:3.8}.
\end{proof}

	In the sequel, by the K\"{a}hler structure of $T\mathcal{E}$, we will implicitly 
	refer to the K\"{a}hler structure of $T\mathcal{E}$ described in Corollary \ref{cor:5.8}.

\begin{example}[\cite{Molitor2012, Molitor-exponential}]\label{exa:5.10} 
	Let $\mathcal{P}_{n}^{\times}$ be the statistical manifold defined in Example \ref{exa:5.5}. 
	For an appropriate normalization of the Fubini-Study metric and symplectic form, it can be shown that
	there exists a map 
	\begin{eqnarray*}
		\tau:T\mathcal{P}_{n}^{\times}\to \mathbb{P}(\mathbb{C}^{n})^{\times},
	\end{eqnarray*}
	where $\mathbb{P}(\mathbb{C}^{n})^{\times}=\{[z_{1},...,z_{n}]\in \mathbb{P}(\mathbb{C}^{n})\,\vert\,
	z_{k}\neq 0\,\,\forall\,k=1,...,n\}$, with the following properties:
	\begin{enumerate}[(i)]
		\item $\tau$ is a universal covering map whose Deck transformation group 
			is isomorphic to $\mathbb{Z}^{n-1}$, 
		\item $\tau$ is holomorphic and locally isometric. 
	\end{enumerate}
	In particular, if $\textup{Deck}(\tau)$ denotes the Deck transformation group of $\tau$, then 
	$T\mathcal{P}_{n}^{\times}/\textup{Deck}(\tau)\cong
	\mathbb{P}(\mathbb{C}^{n})^{\times}$ (isomorphism of K\"{a}hler manifolds).
\end{example}
\begin{example}[\textbf{Binomial distribution} \cite{Molitor-exponential}]\label{exa:5.11} 
	Let $\mathcal{B}(n)$ be the set of binomial distributions defined over $\Omega:=\{0,...,n\}$, 
	as in Example \ref{exa:5.6}. Let $S^{2}$ be the unit sphere in $\mathbb{R}^{3}$. Consider the 
	map $\tau:T\mathcal{B}(n)\to (S^{2})^{\times}:=S^{2}-\{(\pm 1,0,0)\}$ given by 
	\begin{eqnarray*}
		\tau(q,r)=\bigg(\textup{tanh}(q/2), \dfrac{\textup{cos}(r/2)}{\textup{cosh}(q/2)}, 
		\dfrac{\textup{sin}(r/2)}{\textup{cosh}(q/2)}\bigg), 
	\end{eqnarray*}
	where $(q,r)$ are the coordinates on $T\mathcal{B}(n)$ associated to the natural parameter $\theta$, as described 
	before Lemma \ref{nfkewefnknknk}. 

	It is easy to check that if the K\"{a}hler structure of $S^{2}$ (as described in Example \ref{nkdnfkekwnk}) 
	is multiplied by $n$, then $\tau$ is a holomorphic and locally isometric universal covering 
	map whose Deck transformation group is $\textup{Deck}(\tau)\cong \mathbb{Z}$. Therefore
	$T\mathcal{B}(n)/\textup{Deck}(\tau)\cong (S^{2})^{\times}$ (isomorphism of K\"{a}hler manifolds). 
\end{example}
\begin{example}[\textbf{Normal distributions} \cite{Molitor2014}]\label{exa:5.9}
	Let $\mathcal{N}$ be the set of Gaussian distributions, as defined in 
	Example \ref{exa:5.4}. As a complex manifold, $T\mathcal{N}$ is the product $\mathbb{H}\times\mathbb{C}$, 
	where $\mathbb{H}:=\{\tau\in\mathbb{C}\:|\: Im(\tau)>0\}$ is the Poincar\'{e} upper half-plane. 
	The metric of the space $\mathbb{H}\times\mathbb{C}$ is the \textit{K\"{a}hler-Berndt metric $g_{KB}$}, 
	which can be described as follows. If $\tau=u+iv\in\mathbb{H}$ and $z=x+iy\in\mathbb{C}$, then in the coordinates
	$\bigl(u,v,x,y\bigr)$,
	\begin{eqnarray*}
		g_{KB}\bigl(\tau,z\bigr)=
		\begin{pmatrix}
		\frac{v+y^{2}}{v^{3}} & 0 & -\frac{y}{v^{2}} & 0\\
		0 & \frac{v+y^{2}}{v^{3}} & 0 & -\frac{y}{v^{2}}\\
		-\frac{y}{v^{2}} & 0 & \frac{1}{v} & 0\\
		0 & -\frac{y}{v^{2}} & 0 & \frac{1}{v}
	\end{pmatrix}.
	\end{eqnarray*}
	This metric plays an important role in the context of Number Theory, in relation to the so-called Jacobi forms 
	\cite{Berndt1998, Eichler1985}. 
\end{example}
	We end this section with some technical results that we will use in the next section.

	Let $\mathcal{E}$ be an exponential family of dimension $n$ defined over the measure space $(\Omega,dx)$, 
	with elements of the form
	\begin{eqnarray*}
		p\bigl(x;\theta\bigr)=\exp\biggl\{ C(x)+\sum_{i=1}^{n}\theta_{i}F_{i}(x)-\psi(\theta)\biggr\},
	\end{eqnarray*}
	where $C,F_{1},...,F_{n}$ are measurable functions on $\varOmega$,
	$\theta=(\theta_{1},...,\theta_{n})$ is a vector in an open subset $\Theta$ of 
	$\mathbb{R}^{n}$ and where $\psi$ is a function defined on $\Theta$.

	Given $i=1,...,n$, we defined $\eta_{i}:\mathcal{E}\to\mathbb{R}$ by 
	\begin{eqnarray*}
		\eta_{i}\bigl(\theta\bigr)\:=\mathbb{E}_{p_{\theta}}\bigl(F_{i}\bigr)
		=\int_{\Omega}F_{i}\bigl(x\bigr)p\bigl(x;\theta\bigr)dx.
	\end{eqnarray*}
	The functions $\eta_{1},...,\eta_{n}$ are called \textit{expectation parameters}. Note that, 
	if the functions $F_{i}:\Omega\to\mathbb{R}$ are not measurable, the existence of the functions 
	$\eta_{i}$ is not guaranteed. However, in the particular case where $\Omega$ is finite, the functions
	$\eta_{i}$ exist and have good properties, as described in the following result.
\begin{proposition}\label{prop:5.12}
	Let $\mathcal{E}$ be an exponential family defined over a finite set $\Omega=\{x_0,x_1,...,x_m\}$ 
	endowed with the counting measure, with $C,F_1,...,F_n:\Omega\to\mathbb{R}$, $\theta\:\in\:\Theta$ 
	and $\psi:\Theta\to\mathbb{R}$ as above. The following holds. 
	\begin{enumerate}[(i)]
	\item The set $\Theta$ can be taken equal to $\mathbb{R}^{n}$.
	\item $\bigl(\eta_{1},...,\eta_{n}\bigr)$ is a global system of affine coordinates with respect to 
		$\nabla^{(m)}$.
	\item $\dfrac{\partial\psi}{\partial\theta_{i}}=\eta_{i}$ \,\,\,\,\,\,\,for all $i=1,...,n$, 
	\item $\dfrac{\partial^{2}\psi}{\partial\theta_{i}\partial\theta_{j}}
		=\dfrac{\partial\eta_{i}}{\partial\theta_{j}}=h_{F}\Big(\dfrac{\partial}{\partial \theta_{i}}, 
		\dfrac{\partial}{\partial \theta_{j}}\Big),$ \,\,\,\,\,\,\,for all $i,j=1,...,n$.
	\item $\psi(\theta)=\ln\biggl\{\displaystyle\sum_{k=0}^{m}\exp\biggl( C\bigl(x_{k}\bigr)
			+\sum_{j=1}^{n}\theta_{j}F_{j}\bigl(x_{k}\bigr)\biggr)\biggr\}$ \,\,\,\,\,\,\,
		for all $\theta\in\mathbb{R}^{n}$.
	\end{enumerate}
\end{proposition}
\begin{proof}
	See \cite{Amari-Nagaoka}. 
\end{proof}

\section{Constant scalar curvature}\label{fnekwndkefnknk}

	Let $\Omega=\{x_{0},x_{1},...,x_{m}\}$ be a finite set endowed with the counting measure $dx$. 
	Let $\mathcal{E}$ be a 1-dimensional exponential family defined over $(\Omega,dx)$, with elements of the form 
	\begin{eqnarray}\label{eq:6.1}
		p(x;\theta)=\exp\bigl\{C(x)+\theta F(x)-\psi(\theta)\bigr\},
	\end{eqnarray}
	where $C,F:\Omega\to\mathbb{R}$ are functions, $\theta\in\mathbb{R}$ and $\psi:\mathbb{R}\to\mathbb{R}$ is a function. 
	We denote by $\eta:\mathcal{E}\to\mathbb{R}$ the expectation parameter.

	We will use the following notations:
	\begin{itemize}
		\item $F_{i}:= F(x_{i})$,\quad $C_{i}:= C(x_{i})$,\quad $i=0,...,m$,
		\item $F_{min}:= min\{ F_{0},...,F_{m}\}$, \quad $F_{max}:= max\{ F_{0},...,F_{m}\}$,
		\item $I:=\{0,1,...,m\}$,
		\item $I_{min}:=\{ k\in I\:|\: F_{k}=F_{min}\}$,
		\item $I_{max}:=\{ k\in I\:|\: F_{k}=F_{max}\}$.
	\end{itemize}
	Note that $F_{min}\neq F_{max}$ (since the functions $1$ and $F$ are assumed to be linearly independent). 
	Note also that $I_{min}\neq\emptyset$ and $I_{max}\neq\emptyset$.

\begin{lemma}\label{lem:6.1}
	We have
	\begin{eqnarray*}
		\underset{\theta\to-\infty}{\textup{lim}}\eta\bigl(\theta\bigr)= F_{min} 
		\qquad\textup{and}\qquad\underset{\theta\to+\infty}{\textup{lim}}\eta\bigl(\theta\bigr)= F_{max}.
	\end{eqnarray*}
	In particular, $\eta$ is a bounded function.
\end{lemma}
\begin{proof}
	By Proposition \ref{prop:5.12}, (iii) and (iv), we have 
	\begin{eqnarray*}
	\eta\bigl(\theta\bigr) &=& \dfrac{\partial\psi}{\partial\theta}=
	\dfrac{\partial}{\partial\theta}\biggl(\ln\biggl\{\sum_{k=0}^{m}\exp\biggl\{ C\bigl(x_{k}\bigr)
		+\theta F\bigl(x_{k}\bigr)\biggr\}\biggr\}\biggr)\\
	&=& \sum_{k=0}^{m}F_{k}\frac{e^{C_{k}+\theta F_{k}}}{e^{C_{0}+\theta F_{0}}+\cdots+e^{C_{m}+\theta F_{m}}}.
	\end{eqnarray*}
	Multiplying the numerator and denominator by $e^{-\theta F_{min}}$ yields 
	\begin{eqnarray*}
		\eta\bigl(\theta\bigr) &=& \sum_{k=0}^{m}F_{k}\dfrac{e^{C_{k}+\theta(F_{k}-F_{min})}}{e^{C_{0}
		+\theta(F_{0}-F_{min})}+\cdots+e^{C_{m}+\theta(F_{m}-F_{min})}}\\
		&=& \sum_{k\in I_{min}}F_{min}\dfrac{e^{C_{k}}}{\underset{i\in I_{min}}{\sum}e^{C_{i}}
		+\underset{i\in I_{min}^{c}}{\sum}e^{C_{i}+\theta(F_{i}-F_{min})}}\\
	 	&& +\sum_{k\in I_{min}^{c}}F_{k}\dfrac{e^{C_{k}+\theta(F_{k}-F_{min})}}{\underset{i\in I_{min}}{\sum}e^{C_{i}}
		+\underset{i\in I_{min}^{c}}{\sum}e^{C_{i}+\theta(F_{i}-F_{min})}},
	\end{eqnarray*}
	where $I_{min}^{c}=I-I_{min}$. If $k\in I_{min}^{c}$, then $F_{k}-F_{min}>0$ and so, 
	\begin{eqnarray*}
		\underset{\theta\to -\infty}{\textup{lim}}e^{C_{k}+\theta(F_{k}-F_{min})}=0.
	\end{eqnarray*}
	Thus, 
	\begin{eqnarray*}
	\underset{\theta\to -\infty}{\textup{lim}}\eta\bigl(\theta\bigr) 
	= \sum_{k\;\in\; I_{min}}F_{min}\frac{e^{C_{k}}}{\underset{i\;\in\; I_{min}}{\sum}e^{C_{i}}} 
	= F_{min}\dfrac{\underset{k\;\in\; I_{min}}{\sum}e^{C_{k}}}{\underset{i\;\in\; I_{min}}{\sum}e^{C_{i}}} = F_{min}.
	\end{eqnarray*}

	Analogously, 
	\begin{eqnarray*}
		\underset{\theta\to +\infty}{\textup{lim}}\eta\bigl(\theta\bigr)=F_{max}.
	\end{eqnarray*}
	The lemma follows.
\end{proof}

	Let $(g,J,\omega)$ be the K\"{a}hler structure on $T\mathcal{E}$ associated to $\bigl(h_{F},\nabla^{(e)}\bigr)$ 
	via Dombrowski's construction. We denote by $\textup{Scal}: T\mathcal{E}\to\mathbb{R}$ the corresponding scalar curvature.

\begin{proposition}\label{prop:6.2} 
	Suppose the scalar curvature of $T\mathcal{E}$ is constant and equal to $\lambda\in\mathbb{R}$. 
	Then $\lambda\neq0$ and there exist $a,b,r,s\in\mathbb{R}$, with $a\neq b$, such that
	\begin{eqnarray*}
		\psi\bigl(\theta\bigr)=\frac{2}{\lambda}\ln\bigl\{ e^{a\theta+r}+e^{b\theta+s}\bigr\}
	\end{eqnarray*}
	for all $\theta\in\mathbb{R}$. Consequently, the coordinate expression for the Fisher metric 
	with respect to $\theta$ is 
	\begin{eqnarray*}
		 h_{F}(\theta)=\dfrac{\partial^{2}\psi}{\partial \theta^{2}}=
		 \dfrac{(a-b)^{2}}{2\lambda\,\textup{cosh}^{2}\big(\tfrac{a-b}{2}\theta+\tfrac{r-s}{2}\big)},
	\end{eqnarray*}
	where $\textup{cosh}(x)=\tfrac{e^{x}+e^{-x}}{2}$ is the hyperbolic cosine function.
\end{proposition}
\begin{proof}
	By Corollary \ref{cor:4.3}, 
	\begin{eqnarray*}
		\textup{Scal}=-h_F\bigl(\theta\bigr){}^{-1}\frac{\partial^{2}}{\partial\theta^{2}}
		\bigl(\ln\bigl(h_F\bigl(\theta\bigr)\bigr)\bigr),
	\end{eqnarray*}
	where $h_F(\theta):= h_F\bigl(\frac{\partial}{\partial\theta},\frac{\partial}{\partial\theta}\bigr)$, 
	and so, 
	\begin{eqnarray*}
	\textup{Scal}=\lambda &\Leftrightarrow & -h_F\bigl(\theta\bigr){}^{-1}
		\frac{\partial^{2}}{\partial\theta^{2}}\bigl(\ln\bigl(h_F\bigl(\theta\bigr)\bigr)\bigr)=\lambda\\
		&\Leftrightarrow & \frac{\partial^{2}}{\partial\theta^{2}}\bigl(\ln\bigl(h_F\bigl(\theta\bigr)\bigr)\bigr)
			=-\lambda h_F\bigl(\theta\bigr)\\
		&\Leftrightarrow & \frac{\partial^{2}}{\partial\theta^{2}}
			\Bigl(\ln\Bigl(\frac{\partial\eta}{\partial\theta}\Bigr)\Bigr)
			=\frac{\partial}{\partial\theta}\bigl(-\lambda\eta\bigr),
	\end{eqnarray*}
	where we have used $h_{F}=\tfrac{\partial \eta}{\partial \theta}$ (see Proposition \ref{prop:5.12}). 
	Integrating we obtain 
	\begin{eqnarray*}
		\frac{{\displaystyle \partial}}{\partial\theta}\Bigl(\ln\Bigl(\frac{\partial\eta}{\partial\theta}\Bigr)\Bigr)
		=-\lambda\eta+c_{1},\qquad c_{1}\in\mathbb{R},
	\end{eqnarray*}
	which is equivalent to 
\begin{eqnarray*}
	 \dfrac{ \dfrac{\partial^{2}\eta}{\partial\theta^{2}}}{ \dfrac{\partial\eta}{\partial\theta}} 
		=-\lambda\eta+c_{1} 
	 &\Leftrightarrow& \dfrac{\partial^{2}\eta}{\partial\theta^{2}} 
		=\dfrac{\partial}{\partial\theta}\bigl(-\frac{\lambda}{2}\eta^{2}+c_{1}\eta+c_{2}\bigr)\\
 	 &\Leftrightarrow& \dfrac{\partial\eta}{\partial\theta}=-\frac{\lambda}{2}\eta^{2}+c_{1}\eta +c_{2}+c_{3},
\end{eqnarray*}
	where $c_{2},c_{3}\in\mathbb{R}$. We conclude that $\textup{Scal}\equiv\lambda$ 
	if and only if there exist $a,b\in\mathbb{R}$ such that
\begin{eqnarray}\label{eq:6.2}
	\frac{\partial\eta}{\partial\theta}=-\frac{\lambda}{2}\eta^{2}+a\eta+b.
\end{eqnarray}
	Because ${\displaystyle \tfrac{\partial\eta}{\partial\theta}=h_F\bigl(\theta\bigr)>0}$, 
	Equation \eqref{eq:6.2} implies that
\begin{eqnarray}\label{eq:6.3}
	-\frac{\lambda}{2}\eta\bigl(\theta\bigr)^{2}+a\eta\bigl(\theta\bigr)+b>0.
\end{eqnarray}
	Therefore we can divide both sides of \eqref{eq:6.2} by \eqref{eq:6.3}. This yields
	\begin{eqnarray}\label{eq:6.4}
	\dfrac{\dfrac{\partial\eta}{\partial\theta}}{-\frac{\lambda}{2}\eta\bigl(\theta\bigr)^{2}+a\eta\bigl(\theta\bigr)+b}
		=1.
	\end{eqnarray}
	Hence it all boils down to integrate the function 
	\begin{eqnarray*}
		\dfrac{1}{-\frac{1}{2}\lambda x^{2}+ax+b},\qquad x\in\mathbb{R}.
	\end{eqnarray*}
	Let $\triangle=a^{2}-4b(-\tfrac{1}{2}\lambda)=a^{2}+2b \lambda$ 
	be the discriminant of the polynomial $ -\tfrac{1}{2}\lambda x^{2}+ax+b$. We will consider 3 cases.\\

	Case 1: $\triangle<0$. Integration of \eqref{eq:6.4} yields:
	\begin{eqnarray*}
		\dfrac{2}{\sqrt{-\triangle}}\arctan\biggl(\dfrac{-\lambda\eta+a}{\sqrt{-\triangle}}\biggr)=
		\theta+c, \quad c\in\mathbb{R}.
	\end{eqnarray*}
	The left hand side of this equation is a bounded function, whereas the right hand side is not. 
	Thus this case is not possible.

	Case 2: $\triangle=0$. First, suppose that $\lambda=0$. Then the condition 
	$\triangle=0$ implies that $a=0$, which also implies by \eqref{eq:6.3} that $b>0$. 
	On the other hand, it follows from \eqref{eq:6.2} that $\eta\bigl(\theta\bigr)=b\theta+c$. 
	By Lemma \ref{lem:6.1}, the function $\eta$ is bounded, whereas the function $b\theta+c$ is not (since $b\neq0$). 
	It follows that $\lambda=0$ is impossible.

	Now suppose that $\lambda\neq0$. In this case, there exists $\alpha\in\mathbb{R}$ such that
	\begin{eqnarray*}
		-\dfrac{\lambda}{2}\eta^{2}+a\eta+b=-\dfrac{\lambda}{2}\bigl(\eta+\alpha\bigr)^{2}.
	\end{eqnarray*} 
	Integrating \eqref{eq:6.4} we obtain
	\begin{eqnarray}\label{eq:6.19}
		\dfrac{2}{\lambda}\dfrac{1}{\eta+\alpha}=\theta+c,\quad c\in\mathbb{R}.
	\end{eqnarray}
	Putting $\theta=-c$ in \eqref{eq:6.19} yields $\tfrac{1}{\eta(-c)+\alpha}=0$, which is not possible. 

	Case 3: $\triangle>0$. Suppose first that $\lambda=0$ (in particular, this implies $a\neq0$). 
	Then, by \eqref{eq:6.4},
	\begin{eqnarray*}
		\dfrac{\dfrac{\partial\eta}{\partial\theta}}{a\eta(\theta)+b}=1,
	\end{eqnarray*}
	and so, there exists $c\in \mathbb{R}-\{0\}$ such that 
	\begin{eqnarray*}
		 \eta(\theta)=\dfrac{1}{a}(ce^{a\theta}-b),
	\end{eqnarray*}
	which is not possible, since $\eta$ is bounded.

	Suppose that $\lambda\neq0$. In this case, there exist $\alpha,\beta\in\mathbb{R}$ with 
	$\alpha<\beta$, such that
	\begin{eqnarray*}
		-\frac{\lambda}{2}\eta{}^{2}+a\eta+b=-\frac{\lambda}{2}\bigl(\eta-\alpha\bigr)\bigl(\eta-\beta\bigr).
	\end{eqnarray*}
	Then, integration of \eqref{eq:6.4} yields
	\begin{eqnarray*}
		-\frac{2}{\lambda}\frac{1}{\alpha-\beta}\ln\biggl|\frac{\eta-\alpha}{\eta-\beta}\biggr|=\theta+c,\quad c\in\mathbb{R},
	\end{eqnarray*}
	and so
	\begin{eqnarray}
		\biggl|\frac{\eta-\alpha}{\eta-\beta}\biggr|=e^{\frac{\lambda}{2}(\theta+c)(\beta-\alpha)}.\label{eq:6.5}
	\end{eqnarray}
	It follows from \eqref{eq:6.5} that $\alpha,\beta\notin\textup{Im}\bigl(\eta\bigr)$. 
	Therefore we have the following possibilities: 

	(i) $\textup{Im}\bigl(\eta\bigr)\subseteq\bigl(-\infty,\alpha\bigr)\cup\bigl(\beta,+\infty\bigr)$.
	
	\noindent In this case, \eqref{eq:6.5} becomes
	\begin{eqnarray*}
		\dfrac{\eta-\alpha}{\eta-\beta} =e^{\tfrac{\lambda}{2}(\theta+c)(\beta-\alpha)}\,\,\,\,\,\,
		\Leftrightarrow\,\,\,\,\,\,
		\eta-\alpha=\bigl(\eta-\beta\bigr)e^{\tfrac{\lambda}{2}(\theta+c)(\beta-\alpha)}.
	\end{eqnarray*}
	Putting $\theta=-c$ we obtain $\alpha=\beta$ which is a contradiction.
	
	(ii) $\textup{Im}(\eta)\subseteq (\alpha,\beta)$. In this case, 
	\begin{eqnarray*}
		\dfrac{\eta-\alpha}{\beta-\eta}=e^{\tfrac{\lambda}{2}(\theta+c)(\beta-\alpha)},
	\end{eqnarray*}
	from which it follows that 
	\begin{eqnarray*}
		\eta(\theta)=\dfrac{\alpha e^{\tfrac{\lambda}{2}\alpha(\theta+c)}
		+\beta e^{\tfrac{\lambda}{2}\beta(\theta+c)}}{e^{\tfrac{\lambda}{2}\alpha(\theta+c)}
		+e^{\tfrac{\lambda}{2}\beta(\theta+c)}}.
	\end{eqnarray*}
	Integrating again (remember that $\frac{\partial\psi}{\partial\theta}=\eta$) we obtain 
	\begin{eqnarray*} 
		\psi(\theta)=\dfrac{2}{\lambda}\ln\Big( e^{\tfrac{\lambda}{2}\alpha(\theta+c)}
		+e^{\tfrac{\lambda}{2}\beta(\theta+c)}\Big)+\tfrac{2}{\lambda}\ln\bigl(e^{\omega}\bigr),
	\end{eqnarray*}
	where $\tfrac{2}{\lambda}\ln\bigl(e^{\omega}\bigr)$ is a constant. Then,
	\begin{eqnarray*}
		\psi(\theta) &=& \tfrac{2}{\lambda}\ln\big(e^{\tfrac{\lambda}{2}\alpha\theta
			+\tfrac{\lambda}{2}\alpha c}+e^{\tfrac{\lambda}{2}\beta\theta+\tfrac{\lambda}{2}\beta c)}\big)
			+\tfrac{2}{\lambda}\ln\bigl(e^{\omega}\bigr)\\
		 &=& \tfrac{2}{\lambda}\ln\big(\bigl(e^{\tfrac{\lambda}{2}\alpha\theta+\tfrac{\lambda}{2}\alpha c}+
			e^{\tfrac{\lambda}{2}\beta\theta+\tfrac{\lambda}{2}\beta c}\bigr)e^{\omega}\big)\\
 		&=& \tfrac{2}{\lambda}\ln\big(\bigl(e^{\tfrac{\lambda}{2}\alpha\theta+\tfrac{\lambda}{2}\alpha c
			+\omega}+e^{\tfrac{\lambda}{2}\beta\theta+\tfrac{\lambda}{2}\beta c+\omega}\bigr)\big).
	\end{eqnarray*}
	Letting  $a:=\frac{\lambda}{2}\alpha$, $b:=\frac{\lambda}{2}\beta$,
	$r:=\frac{\lambda}{2}\alpha c+\omega$ and $s:=\frac{\lambda}{2}\beta c+\omega$,
	yields the desired result.
\end{proof}
\begin{corollary}
	If the scalar curvature of $T\mathcal{E}$ is constant and equal to $\lambda\in\mathbb{R}$, then $\lambda>0$. 
\end{corollary}
\begin{proof}
	This follows immediately from the formula 
	$h_{F}(\theta)= \tfrac{(a-b)^{2}}{2\lambda\,\textup{cosh}^{2}\big(\tfrac{a-b}{2}\theta+\tfrac{r-s}{2}\big)}$
	and the fact that $h_{F}(\theta)>0$ for all $\theta\in \mathbb{R}$ (since $h_{F}$ is a metric). 
\end{proof}
\begin{remark}\label{nfekwdkefnknk}
	If $\psi(\theta)=\tfrac{2}{\lambda}\ln\bigl\{ e^{a\theta+r}
	+e^{b\theta+s}\bigr\}$, with $a<b$, then $a=\tfrac{\lambda}{2}F_{min}$ and $b=\tfrac{\lambda}{2}F_{max}$. 
	To see this, it suffices to compute $\textup{lim}_{\theta\to \pm\infty}\,\eta(\theta)=
	\textup{lim}_{\theta\to \pm \infty}\tfrac{\partial \psi}{\partial \theta}$ and to compare with Lemma \ref{lem:6.1}. 
\end{remark}

\begin{lemma}\label{lem:6.5} 
	Let $\alpha$ be a nonzero real number and $f:\mathbb{R}\to \mathbb{R},$ $x\mapsto (1+e^{x})^{\alpha}$. 
	Given $k\in \mathbb{N}=\{0,1,...\}$, let $\mathcal{A}_{k}$ be the linear subspace of $C^{\infty}(\mathbb{R})$ 
	spanned by $f$, $f'$,...,$f^{(k)}$ (derivatives of $f$), that is, 
	\begin{eqnarray*}
		\mathcal{A}_{k}=\textup{Span}\bigl\{f,f',...,f^{(k)}\bigr\}.
	\end{eqnarray*}
	If $\alpha\notin\mathbb{N}$, then $\textup{dim}\mathcal{A}_{k}\geq k+1$. 
\end{lemma}
\begin{proof}
	It is easy to see that if $\alpha\not\in \mathbb{N}$, then the family of linearly independent functions 
	\begin{eqnarray*}
		 \phi_{l}(x):=(1+e^{x})^{\alpha-l}e^{lx}, \,\,\,\,\,\,\,\,\,\,0\leq l\leq k, 
	\end{eqnarray*}
	is contained in $\mathcal{A}_{k}$. 
\end{proof}

\begin{proposition}\label{prop:6.6}
	Suppose that the scalar curvature of $T\mathcal{E}$ is constant and equal to $\lambda\in\mathbb{R}$. 
	Then there exists a positive integer $d$ such that $\lambda=\tfrac{2}{d}.$
\end{proposition}
\begin{proof}
	From Proposition \ref{prop:5.12}, there exist
	$a,b,r,s\in\mathbb{R}$ such that 
	\begin{eqnarray*}
		\psi(\theta)=\frac{2}{\lambda}\ln\bigl\{ e^{a\theta+r}+e^{b\theta+s}\bigr\}
	\end{eqnarray*}
	for all $\theta\in\mathbb{R}$. On the other hand, it follows from Proposition \ref{prop:5.12} that 
	\begin{eqnarray*}
		\psi(\theta)=\ln\biggl\{\sum_{k=0}^{m}e^{C_{k}+\theta F_{k}}\biggr\}.
	\end{eqnarray*}
	Thus 
	\begin{eqnarray*}
		\dfrac{2}{\lambda}\ln\bigl\{ e^{a\theta+r}+e^{b\theta+s}\bigr\}=\ln\biggl\{\sum_{k=0}^{m}e^{C_{k}
		+\theta F_{k}}\biggr\}
	\end{eqnarray*}
	that is, 
	\begin{eqnarray*}
		\bigl(e^{a\theta+r}+e^{b\theta+s}\bigr)^{\frac{2}{\lambda}}=\sum_{k=0}^{m}e^{C_{k}+\theta F_{k}}.
	\end{eqnarray*}
	Multiplying by $e^{-\frac{2}{\lambda}(a\theta+r)}$ we obtain 
	\begin{eqnarray*}
		\bigl(1+e^{x}\bigr)^{\alpha}=\sum_{k=0}^{m}e^{\xi_{k}\theta+\omega_{k}},
	\end{eqnarray*}
	where $\alpha:=\frac{2}{\lambda}>0$, $x:=(b-a)\theta+s-r$, 
	$\xi_{k}:= F_{k}-\frac{2}{\lambda}a$ and $\omega_{k}:= C_{k}-\frac{2}{\lambda}r$. 
	Since $\theta=\tfrac{x-(s-r)}{b-a}$, this can be rewritten as
	\begin{eqnarray}
		\bigl(1+e^{x}\bigr)^{\alpha}=\sum_{k=0}^{m}e^{\xi_{k}^{'}x+\omega_{k}^{'}},\label{eq:6.11}
	\end{eqnarray}
	where $\xi_{k}^{'}:=\frac{\xi_{k}}{b-a}$ and $\omega_{k}^{'}:=-\frac{s-r}{b-a}\xi_{k}+\omega_{k}$.
	Note that \eqref{eq:6.11} holds for all $x\in\mathbb{R}$.

	Consider the linear subspace of $C^{\infty}(\mathbb{R})$ spanned by the functions 
	$e^{\xi_{k}^{'}x+\omega_{k}^{'}}$, $k=0,...,m$, that is,
	\begin{eqnarray*}
		E:= \textup{Span}\Big\{ e^{\xi_{k}^{'}x+\omega_{k}^{'}}\:\Big|\: k=0,...,m\Big\}.
	\end{eqnarray*}
	Observe that $\textup{dim}E\leq m+1$ and that for every $h\in E$, the derivative of $h$ with respect to $x$ belongs to 
	$E$, that is, $\frac{dh}{dx}\in E$.

	Let $f:\mathbb{R}\to\mathbb{R}$ be the function defined by $f(x):=(1+e^{x})^{\alpha}$. Because 
	of \eqref{eq:6.11}, $f$ belongs to $E$, and by the observation above, so does its derivatives of all orders. 
	Therefore $\mathcal{A}_{k}:=\textup{Span}\{f,f',f'',...,f^{(k)}\}$ is a linear subspace of $E$ 
	for every integer $k\geq0$, which implies $\textup{dim}\,\mathcal{A}_{k}\leq m+1$ for every $k$. 
	According to Lemma \ref{lem:6.5}, this is only possible if $\alpha\in\mathbb{N}$, that is, 
	if $\frac{2}{\lambda}\in\mathbb{N}$.
\end{proof}

	In what follows, we will use the following notations:

	\begin{itemize}
	\item $\alpha_{0},...,\alpha_{p}$ are the unique real numbers such that 
		$\alpha_{0}<...<\alpha_{p}$ and $\textup{Im}(F)=\{\alpha_{0},...,\alpha_{p}\},$
	\item $I_{i}:=\{ k\;\in\;I\:|\: F_{k}=\alpha_{i}\}$,
		$i=0,...,p$,
	\item $e^{\omega_{i}}=\underset{k\;\in\; I_{i}}{\sum}e^{C_{k}}$.
	\end{itemize}
	Note that $p\leq m$ and that $\alpha_{0}=F_{min}$ 
	and $\alpha_{p}=F_{max}$.
\begin{lemma}\label{ndsndkdvnsn}
	Assume that $\textup{Scal}\equiv\frac{2}{d}$ on $T\mathcal{E}$, 
	with $d\in\mathbb{N}^{*}$. With the notation of Proposition \ref{prop:6.6}, we have
	\begin{eqnarray*}
		\sum_{k=0}^{p}e^{\theta\alpha_{k}+\omega_{k}}=
		\sum_{l=0}^{d}e^{\big[\alpha_{0}+\tfrac{l}{d}(\alpha_{p}-\alpha_{0})\big]\theta+rl+s(d-l)+\ln\binom{d}{l}}
	\end{eqnarray*}
	for every $\theta\in\mathbb{R}$.
\end{lemma}
\begin{proof}
	We know from Proposition \ref{prop:6.2} and Proposition \ref{prop:6.6} that 
	there are real numbers $a,b,r,s$, with $a<b$, and 
	$d\in \mathbb{N}^{*}$ such that 
	\begin{eqnarray*}
		 \psi(\theta)=d\ln\big(e^{a\theta+r}+e^{b\theta+s}\big)=d\ln\Big(e^{\tfrac{\alpha_{0}}{d}\theta+r}
		+e^{\tfrac{\alpha_{p}}{d}\theta+s}\Big), 
	\end{eqnarray*}
	for all $\theta\in \mathbb{R}$, where we have used $a=\tfrac{\lambda}{2}F_{min}=\tfrac{\alpha_{0}}{d}$ 
	and $b=\tfrac{\lambda}{2}F_{max}=\tfrac{\alpha_{p}}{d}$ (see Remark \ref{nfekwdkefnknk}). 
	From Proposition \ref{prop:5.12}, we also have that 
	$\psi(\theta)=\ln\big(\sum_{k=0}^{m}e^{C_{k}+\theta F_{k}}\big)$. Therefore 
\begin{eqnarray}
	\Big(e^{\tfrac{\alpha_{p}}{d}\theta+r}+e^{\tfrac{\alpha_{0}}{d}\theta+s}\Big)^{d}
	=\sum_{k=0}^{m}e^{C_{k}+\theta F_{k}}.\label{eq:6.12}
\end{eqnarray}
	We compute the left and right hand sides of \eqref{eq:6.12} separately:
\begin{eqnarray*}
	\textup{left hand side} &=& \sum_{l=0}^{d}\binom{d}{l}\Bigl(e^{\tfrac{\alpha_{p}}{d}\theta+r}\Bigr)^{l}\Bigl(
		e^{\tfrac{\alpha_{0}}{d}\theta+s}\Bigr)^{d-l} \\
 	&=&\sum_{l=0}^{d}e^{\big[\alpha_{0}+\tfrac{l}{d}(\alpha_{p}-\alpha_{0})\big]\theta+rl+s(d-l)+\ln\binom{d}{l}},\\
	\textup{right hand side} &=&  \sum_{k\;\in\; I_{0}}^{m}e^{C_{k}+\theta\alpha_{0}}+\cdots
		+\sum_{k\;\in\; I_{p}}^{m}e^{C_{k}+\theta\alpha_{p}}\\
 	&=&	e^{\theta\alpha_{0}}\Biggl(\sum_{k\;\in\; I_{0}}^{m}e^{C_{k}}\Biggr)
		+\cdots+e^{\theta\alpha_{p}}\Biggl(\sum_{k\;\in\; I_{p}}^{m}e^{C_{k}}\Biggr)\\
	&=&   e^{\theta\alpha_{0}+\omega_{0}}+\cdots+e^{\theta\alpha_{p}+\omega_{p}}, 
\end{eqnarray*}
	where we have used $e^{\omega_{i}}=\underset{k\;\in\; I_{i}}{\sum}e^{C_{k}}$. The lemma follows by comparing the 
	left and right hand sides. 
%
%
\end{proof}

\begin{lemma}\label{lem:6.8} 
	Let $\{\xi_{i},\eta_{i}\,\,\vert\,\,i=0,...,d\}$ and  $\{\alpha_{j},\omega_{j}\,\,\vert\,\,j=0,...,p\}$ be families of 
	real numbers such that $\xi_{0}<\cdots<\xi_{d}$ and $\alpha_{0}<\cdots<\alpha_{p}$. If 
	\begin{eqnarray*}
		e^{\xi_{0}\theta+\eta_{0}}+\cdots+e^{\xi_{d}\theta+\eta_{d}}=
		e^{\alpha_{0}\theta+\omega_{0}}+\cdots+e^{\alpha_{p}\theta+\omega_{p}}
	\end{eqnarray*}
	for all $\theta\in\mathbb{R}$, then $d=p$ and for every $i=0,...,d$, $\xi_{i}=\alpha_{i}$ and $\eta_{i}=\omega_{i}$.
\end{lemma}
\begin{proof}
	By hypothesis, we have 
	\begin{eqnarray*}
		1=\dfrac{e^{\xi_{0}\theta+\eta_{0}}+\cdots+e^{\xi_{d}\theta+\eta_{d}}}
		{e^{\alpha_{0}\theta+\omega_{0}}+\cdots+e^{\alpha_{p}\theta+\omega_{p}}}=
		\dfrac{F(\theta)e^{\xi_{d}\theta+\eta_{d}}}{G(\theta)e^{\alpha_{p}\theta+\omega_{p}}}=
		\dfrac{F(\theta)}{G(\theta)}
		e^{(\xi_{d}-\alpha_{p})\theta+(\eta_{d}-\omega_{p})}
	\end{eqnarray*}
	for all $\theta\in \mathbb{R}$, where $F,G:\mathbb{R}\to \mathbb{R}$ are functions that are easily seen 
	to satisfy $\textup{lim}_{\theta\to \infty}F(\theta)=1$ and $\textup{lim}_{\theta\to \infty}G(\theta)=1$.
	It follows that 
	\begin{eqnarray*}
		\underset{\theta\to \infty}{\textup{lim}}\,e^{(\xi_{d}-\alpha_{p})\theta+(\eta_{d}-\omega_{p})}=1, 
	\end{eqnarray*}
	which forces $\xi_{d}=\alpha_{p}$ and $\eta_{d}=\omega_{p}$. The lemma is proved by repeating the same argument. 
\end{proof}

\begin{theorem}\label{nekwndkeknskn}
	Let $\mathcal{E}$ be a 1-dimensional exponential family defined over a finite set $\Omega=\{x_{0},...,x_{m}\}$, 
	with elements of the form $p(x;\theta)=\textup{exp}\{C(x)+\theta F(x)-\psi(\theta)\}$, where 
	$C,F:\Omega\to \mathbb{R}$, $\theta\in \mathbb{R}$ and $\psi:\mathbb{R}\to \mathbb{R}$. Suppose that 
	$\textup{Im}(F)=\{\alpha_{0}<...<\alpha_{p}\}$. Given $i=0,...,p$, define $\omega_{i}\in \mathbb{R}$ via the formula 
		\begin{eqnarray*}
			e^{\omega_{i}}=\underset{k\in F^{-1}(\alpha_{i})}{\sum}e^{C(x_{k})}.
		\end{eqnarray*}
	Then the scalar curvature $\textup{Scal}:T\mathcal{E}\to \mathbb{R}$ is constant
	if and only if there exist $r,s\in\mathbb{R}$ such that
	\begin{eqnarray}
	\left\lbrace
	\begin{array}{l}\label{eq:6.18}
		\alpha_{k}=\alpha_{0}+\frac{k}{p}\bigl(\alpha_{p}-\alpha_{0}\bigr),\\
		\omega_{k}=rk+s\bigl(p-k\bigr)+\ln\binom{p}{k},
	\end{array}
	\right.
	\end{eqnarray}
	for all $k=0,...,p$, and in that case, $\textup{Scal}\equiv\frac{2}{p}$.
\end{theorem}
\begin{proof}
	$(\Rightarrow)$ This follows from Lemma \ref{ndsndkdvnsn} and Lemma \ref{lem:6.8}. $(\Leftarrow)$ This can be proved 
	by reversing the reasoning above. 
\end{proof}

\begin{corollary}\label{ekwjdkefkjk}
	Let $\mathcal{E}$ be a 1-dimensional exponential family defined over a finite set 
	$\Omega=\bigl\{ x_{0},...,x_{m}\bigr\}$. If the scalar curvature of $T\mathcal{E}$ 
	is constant and equal to $\lambda$, then
	\begin{eqnarray*}
		\lambda\in\Bigl\{\frac{2}{k}\;\Big\vert\;1\leq k\leq m\Bigr\}.
	\end{eqnarray*}
\end{corollary}

\begin{example}\label{nnkqwnkenkwnk}
\textbf{(Binomial distribution)}. Recall that elements of $\mathcal{B}(n)$ are parametrized as follows 
	\begin{eqnarray*}
		p(k;\theta)=\binom{n}{k}q^{k}\bigl(1-q\bigr)^{n-k}=
		\textup{exp}\bigg\{\ln\binom{n}{k} +\theta k-n\ln(1+e^{\theta}\bigr)\bigg\},
	\end{eqnarray*}
	where $k\in\bigl\{0,...,n\bigr\}$ and $\theta=\ln(\tfrac{q}{1-q})\in \mathbb{R}$. In this case, we have $p=n$ and 
	\begin{eqnarray*}
		 \alpha_{k}=k \,\,\,\,\,\,\,\,\,\,\,\,\,\,\,\,\textup{and} \,\,\,\,\,\,\,\,\,\,\,\,\,\,\,\,
		\omega_{k}=\ln\binom{n}{k}
	\end{eqnarray*}
	for all $k=0,...,n$. Clearly $\alpha_{0},\alpha_{1},...,\alpha_{n}$ and $\omega_{0},...,\omega_{n}$ are solutions 
	of \eqref{eq:6.18} with $r=s=0$. Therefore the scalar curvature of $T\mathcal{B}(n)$ is constant and equal 
	to $\frac{2}{n}$.
\end{example}

\section{Equivalent and reduced exponential families}\label{kwdenfknknk}

	The following notation will be used throughout this section. Given a finite set $\Omega=\{x_{0},...,x_{m}\}$, let 
	$C(\Omega)$ denote the space of maps $\Omega\to \mathbb{R}$ (clearly there is a natural identification 
	$C(\Omega)\cong \mathbb{R}^{m+1}$). Given $F,C\in C(\Omega)$, let $\mathcal{E}_{C,F}$ 
	denote the 1-dimensional exponential family defined over $\Omega$ with elements of the form 
	\begin{eqnarray*}
		 p_{C,F}(x;\theta)=\textup{exp}\big\{C(x)+\theta F(x)-\psi_{C,F}(\theta)\big\},
	\end{eqnarray*}
	where $x\in \Omega$, $\theta\in \mathbb{R}$ and $\psi_{C,F}(\theta)=\ln\big(\sum_{k=0}^{m}\textup{exp}(C(x_{k})
	+\theta F(x_{k}))\big).$ In the above notation, it is assumed that the function $F$ and the constant function 
	$\mathds{1}:\Omega\to \mathbb{R},$ $x\mapsto 1$ are linearly independent (this guarantees that the map 
	$\mathbb{R}\to \mathcal{E}_{C,F}$, $\theta\mapsto p_{C,F}(\,.\,;\theta)$ is bijective). In other words, 
	$F\in \mathbb{R}^{m+1}-\mathbb{R}\cdot \mathds{1}$. 
\begin{definition}\label{kwwdefksn}
	Two 1-dimensional exponential families $\mathcal{E}_{C,F}$ and $\mathcal{E}_{C',F'}$ defined over the same 
	set $\Omega=\{x_{0},...,x_{m}\}$ are \textit{equivalent} if the families of maps 
	$\{p_{C,F}(\,.\,;\theta)\,:\,\Omega\to \mathbb{R}\}_{\theta\in \mathbb{R}}$ and 
	$\{p_{C',F'}(\,.\,;\theta)\,:\,\Omega\to \mathbb{R}\}_{\theta\in \mathbb{R}}$ coincide. 
\end{definition}
	
	In order to caracterize equivalent exponential families, we introduce the group of matrices 
	\begin{eqnarray*}
		G:=\bigg\{
		 \begin{bmatrix}
			1 & b & d\\
			0 & a & c\\
			0 & 0 & 1
		\end{bmatrix}
		\,\,\,\,\bigg\vert\,\,\,\, a,b,c,d\in \mathbb{R}, \,\,a\neq 0 \bigg\}.
	\end{eqnarray*}
	Given an integer $m\geq 1$, the group $G$ acts on 
	$U_{m}:=\mathbb{R}^{m+1}\times (\mathbb{R}^{m+1}-\mathbb{R}\cdot\mathds{1})$ via the formula
	\begin{eqnarray*}
		  \begin{bmatrix}
			1 & b & d\\
			0 & a & c\\
			0 & 0 & 1
		\end{bmatrix}
		\cdot (C,F):= (C+bF+d\mathds{1}, aF+c\mathds{1}), 
	\end{eqnarray*}
	where $C\in \mathbb{R}^{m+1}$ and $F\in \mathbb{R}^{m+1}-\mathbb{R}\cdot\mathds{1}$. 
\begin{proposition}\label{nknknkn}
	Two 1-dimensional exponential families $\mathcal{E}_{C,F}$ and $\mathcal{E}_{C',F'}$ defined over the same 
	finite set $\Omega$ are equivalent if and only if there exists $g=\Big[\begin{smallmatrix}
		1&b&d\\
		0&a&c\\
		0&0&1
	\end{smallmatrix} \Big]\in G$ such that $(C,F)=g\cdot (C',F')$. 
	In that case, the following holds. 
	\begin{enumerate}[(i)]
	\item $p_{C,F}(x;\theta)=p_{C',F'}(x;a\theta+b)$ \,\,\,\,\,for all $x\in \Omega$ and all $\theta\in \mathbb{R}$. 
	\item $\psi_{C,F}(\theta)=\psi_{C',F'}(a\theta+b)+c\theta+d$ \,\,\,\,\,for all $\theta\in \mathbb{R}$. 
	\end{enumerate}
\end{proposition}
\begin{proof}	
	$(\Rightarrow)$ Suppose $\mathcal{E}_{C,F}\sim\mathcal{E}_{C',F'}$. Because the maps $\mathbb{R}\to \mathcal{E}_{C,F}$, 
	$\theta\mapsto p_{C,F}(\,.\,;\theta)$ and 
	$\mathbb{R}\to \mathcal{E}_{C',F'}$, $\theta'\mapsto p_{C',F'}(\,.\,;\theta')$ are bijective, there exists a bijection 
	$\phi:\mathbb{R}\to \mathbb{R}$ such that 
	\begin{eqnarray}\label{kndkefknkn}
		p_{C,F}(x;\theta)=p_{C',F'}(x;\phi(\theta))		 
	\end{eqnarray}
 	for all $\theta\in \mathbb{R}$ and all $x\in \Omega.$ Since $F'$ and $\mathds{1}$ are linearly independent, there 
	exist $y,z\in \Omega$ such that $F'(y)\neq F'(z)$. Putting $x=y$ and $x=z$ 
	in \eqref{kndkefknkn}, we obtain the following system 
	\begin{eqnarray*}
		 \left\lbrace
		\begin{array}{lll}
			C(y)+\theta F(y)-\psi_{C,F}(\theta)=C'(y)+\phi(\theta) F'(y)-\psi_{C',F'}(\phi(\theta)),\\
			C(z)+\theta F(z)-\psi_{C,F}(\theta)=C'(z)+\phi(\theta) F'(z)-\psi_{C',F'}(\phi(\theta)).
		\end{array}
		\right.  
	\end{eqnarray*}
	Subtracting, we obtain 
	\begin{eqnarray*}
		 \phi(\theta)=\theta\dfrac{F(y)-F(z)}{F'(y)-F'(z)}+\dfrac{C(y)-C(z)-(C'(y)-C'(z))}{F'(y)-F'(z)}
	\end{eqnarray*}
	for all $\theta\in \mathbb{R}$. Therefore there exist $a,b\in \mathbb{R}$ such 
	that $\phi(\theta)=a\theta+b$ for all $\theta\in \mathbb{R}$. Note 
	that $a$ is necessarily nonzero. Taking the derivative in \eqref{kndkefknkn} with respect to $\theta$ and using the 
	formula $\phi(\theta)=a\theta+b$ we find that  
	\begin{eqnarray*}
		 F(x)-aF'(x)=\dfrac{d\psi_{C,F}}{d\theta}(\theta)-a\dfrac{d\psi_{C',F'}}{d\theta'}(a\theta+b)
	\end{eqnarray*}
	for all $\theta\in \mathbb{R}$ and all $x\in \Omega.$ This implies that there exists $c\in \mathbb{R}$ such that 
	\begin{eqnarray}\label{nwknkwnkdnkan}
		F(x) =aF'(x)+c
	\end{eqnarray}
	for all $x\in \Omega$, and 
	\begin{eqnarray*}
		 \dfrac{d\psi_{C,F}}{d\theta}(\theta)=a\dfrac{d\psi_{C',F'}}{d\theta'}(a\theta+b)+c
	\end{eqnarray*}
	for all $\theta\in \mathbb{R}$. Integrating the equation above, we obtain 
	\begin{eqnarray}\label{ndwkndkknkdn}
		 \psi_{C,F}(\theta)=\psi_{C',F'}(a\theta+b)+c\theta+d
	\end{eqnarray}
	for all $\theta\in \mathbb{R}$, where $d\in \mathbb{R}$ is some constant. 
	Then, using \eqref{kndkefknkn}, \eqref{nwknkwnkdnkan} and \eqref{ndwkndkknkdn} we see that 
	\begin{eqnarray}\label{nfkndefksnk}
		C(x) =C'(x)+bF'(x)+d
	\end{eqnarray}
	for all $x\in \Omega$. It follows from \eqref{nwknkwnkdnkan} and 
	\eqref{nfkndefksnk} that $(C,F)= \Big[\begin{smallmatrix}
		1&b&d\\
		0&a&c\\
		0&0&1
	\end{smallmatrix} \Big]\cdot (C',F')$, which concludes one direction of the proof. 
	Note that the computation above shows that if $\mathcal{E}_{C,F}\sim \mathcal{E}_{C',F'}$, then (i) and (ii) hold. 

	$(\Leftarrow)$ Left as a simple exercice to the reader. 
\end{proof}

\begin{definition}
	Let $V$ be a finite dimensional real vector space and let $k$ be an integer satisfying $1\leq k\leq \textup{dim}\,V$. 
	The \textit{affine Grassmannian}, denoted by $\textup{Graff}_{k}(V)$, is the set of all $k$-dimensional affine 
	subspaces of $V$. 
\end{definition}
	It can be shown that $\textup{Graff}_{k}(V)$ is a noncompact smooth manifold of dimension $(n-k)(k+1)$, 
	where $n=\textup{dim}\,V$ (see \cite{Lim2019}). 

	Given an integer $m\geq 1$, the space $\mathbb{R}^{m+1}$ decomposes as the following direct sum:
	\begin{eqnarray*}
		 \mathbb{R}^{m+1}=V_{m}\oplus \mathbb{R}\cdot \mathds{1}, 
	\end{eqnarray*}
	where $V_{m}$ is the orthogonal complement of 
	$\mathds{1}=(1,...,1)$ in $\mathbb{R}^{m+1}$ with respect to the usual 
	inner product $\langle,\rangle$ on $\mathbb{R}^{m+1}$, that is, 
	\begin{eqnarray*}
		 V_{m}=\{u\in \mathbb{R}^{m+1}\,\,\vert\,\,\langle u,\mathds{1}\rangle=0\}. 
	\end{eqnarray*}
	Given $u\in \mathbb{R}^{m+1}$, we will denote by $u^{\perp}\in V_{m}$ the orthogonal projection of 
	$u$ on $V_{m}$. 

	Finally, given $(C,F)\in U_{m}$, we will denote by $[C,F]$ the corresponding equivalence class in the quotient space 
	$U_{m}/G$. 

\begin{proposition}\label{nekwndknk}
	For every integer $m\geq 1$, the map 
	\begin{eqnarray*}
		f:U_{m}/G\to \textup{Graff}_{1}(V_{m})		 
	\end{eqnarray*}
	given by $f([C,F]):=C^{\perp}+\textup{span}\{F^{\perp}\}$ is a bijection. 
\end{proposition}
\begin{proof}
	By a direct verification.
\end{proof}
	It follows from Proposition \ref{nknknkn} and Proposition \ref{nekwndknk} that the set of equivalence classes 
	of 1-dimensional exponential families defined over the same finite set $\Omega=\{x_{0},...,x_{m}\}$ 
	is in one-to-one correspondence with $\textup{Graff}_{1}(V_{m})$. 
\begin{example}
	All 1-dimensional exponential families defined over $\Omega=\{x_{0},x_{1}\}$ are equivalent, because
	$\textup{Graff}_{1}(V_{1})=\{V_{1}\}$ is a single point. 
\end{example}

\begin{definition}\label{nckwdknknk}
	Let $\mathcal{E}=\mathcal{E}_{C,F}$ be a 1-dimensional exponential family defined over a finite set 
	$\Omega=\{x_{0},...,x_{m}\}$. Let $\{\alpha_{i}\}_{i=0,...,p}$ and $\{\omega_{i}\}_{i=0,...,p}$ be the families 
	of real numbers characterized by the following conditions: 
	\begin{enumerate}[(i)]
		\item $\textup{Im}\,F=\{\alpha_{0},...,\alpha_{p}\}$ and $\alpha_{0}<...<\alpha_{p}$, 
		\item $e^{\omega_{i}}=\sum_{k\in F^{-1}(\alpha_{i})}e^{C(x_{k})}$. 
	\end{enumerate}
	Let $\Omega_{red}=\{0,1,...,p\}$. Define $C_{red},F_{red}:\Omega_{red}\to \mathbb{R}$ by 
	\begin{eqnarray*}
		F_{red}(k)=\alpha_{k} \,\,\,\,\,\,\,\textup{and} \,\,\,\,\,\,\, C_{red}(k)=\omega_{k},
	\end{eqnarray*}
	where $k=0,...,p$. Then $\mathcal{E}_{red}:=\mathcal{E}_{C_{red},F_{red}}$ is a 1-dimensional 
	exponential family defined over $\Omega_{red}$. We call it the \textit{reduced exponential family} of 
	$\mathcal{E}$. 
\end{definition}
\begin{remark}
	If $\mathcal{E}=\mathcal{E}_{C,F}$ is a 1-dimensional exponential family defined over a finite set $\Omega$, 
	then $\psi_{C,F}(\theta)=\psi_{C_{red},F_{red}}(\theta)$ for all $\theta\in \mathbb{R}$. 
\end{remark}
\begin{proposition}\label{nkdnkenkknk}
	Let $\mathcal{E}=\mathcal{E}_{C,F}$ be a 1-dimensional exponential family defined over a finite set 
	$\Omega=\{x_{0},...,x_{m}\}$. The following are equivalent.
	\begin{enumerate}[(i)]
	\item The scalar curvature of $T\mathcal{E}$ is constant. 
	\item $\mathcal{E}_{red}\sim \mathcal{B}(p)$, where $p+1$ is the cardinality of $\Omega_{red}$. 
	\end{enumerate}
\end{proposition}
\begin{proof}
	Let $\Omega_{B}=\{0,1,...,p\}$ and let $C_{B},F_{B}\,:\,\Omega_{B}\to \mathbb{R}$ be defined 
	by $C_{B}(k)=\ln\binom{p}{k}$ and $F_{B}(k)=k.$ Comparing with Example \ref{exa:5.6}, we see that 
	$\mathcal{E}_{C_{B},F_{B}}\sim \mathcal{B}(p)$.

	If the scalar curvature of $T\mathcal{E}$ is constant, then by Theorem \ref{nekwndkeknskn} 
	there are real numbers $r$ and $s$ such that 
	\begin{eqnarray}\label{nkandkfekwnwk}
		 \left\lbrace
		\begin{array}{lll}
			F_{red}(k)=F_{red}(0)+\tfrac{k}{p}(F_{red}(p)-F_{red}(0)),\\
			C_{red}(k)=rk+s(p-k)+\ln\binom{p}{k},
		\end{array}
		\right.
	\end{eqnarray}
	for all $k=0,...,p$, which implies that 
	\begin{eqnarray*}
		(C_{red},F_{red})=
		 \begin{bmatrix}
			1 & r-s  &   sp \\
			0 & \tfrac{F_{red}(p)-F_{red}(0)}{p}	& F_{red}(0)  \\
			0 &                  0                  & 1
		\end{bmatrix}
		\cdot (C_{B},F_{B}).
	\end{eqnarray*}
	It follows from this and Proposition \ref{nknknkn} that $\mathcal{E}_{red}\sim \mathcal{E}_{C_{B},F_{B}}\sim \mathcal{B}(p)$. 

	Conversely, if $\mathcal{E}_{red}\sim \mathcal{B}(p)$, then $\mathcal{E}_{red}\sim \mathcal{E}_{C_{B},F_{B}}$ and hence 
	there is $g=\Big[\begin{smallmatrix}
		1&b&d\\
		0&a&c\\
		0&0&1
	\end{smallmatrix} \Big]$ such that $(C_{red},F_{red})=g\cdot (C_{B},F_{B})$, which implies that 
	$C_{red}(k)$ and $F_{red}(k)$ are solutions of \eqref{nkandkfekwnwk} for all $k=0,...,p$, provided 
	$r=b+\tfrac{d}{p}$ and $s=\tfrac{d}{p}$. By Theorem \ref{nekwndkeknskn} again, this implies that 
	the scalar curvature of $T\mathcal{E}$ is constant.
\end{proof}

\begin{remark}
	As we saw in this paper, if the scalar curvature $\textup{Scal}:T\mathcal{E}\to \mathbb{R}$ 
	of the tangent bundle of a 1-dimensional exponential family defined over a finite set is constant, 
	then $\textup{Scal}>0$. This is not true for more general exponential families. For example, 
	if $\mathcal{E}=\mathcal{N}$ is the family of Gaussian distributions over $\mathbb{R}$ (see Example \ref{exa:5.4}), 
	then $\textup{Scal}:T\mathcal{N}\to \mathbb{R}$ is constant and equal to $-6$ 
	(see \cite{Molitor2014}).
\end{remark}

\section*{Acknowledgments}

	I am thankful to Caroline Santos Leite Ribeiro who carefully read and helped typing a 
	preliminary version of this article.


\end{document}